\documentclass[11pt]{amsart}
\usepackage{amsmath}
\usepackage{amssymb}
\usepackage{fancybox}
\usepackage{eucal}
\usepackage{amsthm}
\usepackage{amscd}
\usepackage{hyperref}
\usepackage{wasysym}
\usepackage{url}
\usepackage{enumitem}
\usepackage{amssymb,}
\usepackage[]{amsmath, amsthm, amsfonts,graphicx, amscd,}
\usepackage[all,cmtip]{xy}
\usepackage{xcolor}
\usepackage{soul}
\usepackage{comment} 
\input amssym.def \input amssym

\newtheorem {thm}{Theorem}[section]

\newtheorem{prop}[thm]{Proposition}
\newtheorem{fact}[thm]{Fact}
\newtheorem {lem}[thm]{Lemma}

\newtheorem{ques}[thm]{Question}

\newtheorem{rem}[thm]{Remark}
\theoremstyle{definition}
 
\newtheorem{defn}[thm]{Definition}

\def\Ind{\setbox0=\hbox{$x$}\kern\wd0\hbox to 0pt{\hss$\mid$\hss} \lower.9\ht0\hbox to 0pt{\hss$\smile$\hss}\kern\wd0}
\def\Notind{\setbox0=\hbox{$x$}\kern\wd0\hbox to 0pt{\mathchardef \nn=12854\hss$\nn$\kern1.4\wd0\hss}\hbox to 0pt{\hss$\mid$\hss}\lower.9\ht0 \hbox to 0pt{\hss$\smile$\hss}\kern\wd0}

\newcommand{\m}{\mathbb }
\newcommand{\mc}{\mathcal }

\begin{document}

\title[Algebraic relations of Painlev\'e equations]{Algebraic relations between solutions of Painlev\'e equations}
\author{James Freitag and Joel Nagloo}
\address{James Freitag\\ Department of Mathematics, Statistics, and Computer Science\\
	University of Illinois at Chicago\\
	322 Science and Engineering Offices (M/C 249)\\
	851 S. Morgan Street\\
	Chicago, IL 60607-7045}
\email{freitagj@gmail.com}
\address{Joel Nagloo\\ Department of Mathematics, Statistics, and Computer Science\\
	University of Illinois at Chicago\\
	322 Science and Engineering Offices (M/C 249)\\
	851 S. Morgan Street\\
	Chicago, IL 60607-7045}
\email{jnagloo@uic.edu}

\date{\today}

\begin{abstract}
In this manuscript we make major progress classifying algebraic relations between solutions of Painlev\'e equations. Our main contribution is to establish the algebraic independence of solutions of various pairs of equations in the Painlev\'e families; for generic coefficients, we show all algebraic relations between solutions of equations in the same Painlev\'e family come from classically studied B{\"a}cklund transformations. We also apply our analysis of ranks to establish some transcendence results for pairs of Painlev\'e equations from different families. In that area, we answer several open questions of Nagloo (2016), and in the process answer a question of Boalch (2012). 

We calculate model theoretic ranks of all Painlev\'e equations in this article, extending results of Nagloo and Pillay (2017). We show that the type of the generic solution of any equation in the second Painlev\'e family is geometrically trivial, extending a result of Nagloo (2015). We give the first model theoretic analysis of several special families of the third Painlev\'e equation, proving results analogous to Nagloo and Pillay (2017). We also give a novel new proof of the irreducibility of the third, fifth and sixth Painlev\'e equations using recent work of Freitag, Jaoui, and Moosa (2022). Our proof is fundamentally different than the existing transcendence proofs of Watanabe (1998) or Cantat and Loray (2009).

\end{abstract}
\maketitle

\thanks{J. Freitag is partially supported by NSF CAREER award 1945251 and NSF grant DMS-1700095. J. Nagloo is partially supported by NSF grant DMS-2203508.}

\tableofcontents

\section{Introduction and notation} 
The Painlev\'e equations are six families of nonlinear second order differential equations over $\m C(t) $ with complex parameters. For instance, the second Painlev\'e equation with parameter $\alpha$, $P_{II} ( \alpha) $, is given by\footnote{Throughout our manuscript, $y' := \frac{dy}{dt}$.}: 
\[ P_{II} (\alpha) : \, \, y'' = 2y^3 +ty + \alpha \] where $\alpha \in \m C.$

These equations were isolated over one hundred years ago in the course of answering foundational problems involving analytic continuation. From the beginning of the theory, there was also significant interest in understanding to what extent the solutions to the equations could be expressed in terms of classical or algebraic functions. This question was given a complete answer through the work of the Japanese school of differential algebra\footnote{Specific citations will be given later in the introduction as well as throughout this paper, as we  use this work extensively.}, where the questions were formalized in terms of transcendental/algebraic aspects of a solution of a Painlev\'e equation over $\m C(t)$. 

The Painlev\'e equations have been extensively studied in part because of their relation to physical systems as well as their connection to algebraic geometry and number theory. For instance, $P_{III}$ is connected to the Ising and ferromagnetic models \cite{mccoy1977two, wu1976spin}. Electrodiffusion has been studied using $P_{II}$ \cite{bracken2012Backlund}. $P_{II} $ and $P_{III}$ are connected to behaviour of rogue waves \cite{bilman2020extreme}.
$P_V$ is used in the impenetrable Bose gas model \cite{jimbo1980density}. Scattering on two Aharonov–Bohm vortices can be exactly solved using solutions of $P_{III}$ \cite{bogomolny2016scattering}. See \cite{alves2018solutions}
for additional uses is physics. Solutions of $P_{VI}$ are related to problems around stable vector bundles on curves \cite{hitchin1992poncelet}, while transformations of $P_{VI}$ are related to Markoff triples \cite{bourgain2016markoff}. 

Classifying algebraic relations between Painlev\'e equations is of interest in applications because there is extensive work in concretely expressing solutions of physical systems (or systems from geometry and number theory) in terms of classical functions and solutions to known differential equations using rational or perhaps algebraic formulas - see the physics citations above and \cite{Manin6} for a number theoretic example. The classification of algebraic relations between Painlev\'e equations would be a significant tool in understanding when such expressions are possible. For instance, \cite[page 20]{alves2018solutions} mentions the interest in extending work of \cite{nagloo2015transformations} to additional families of Painlev\'e equations. This extension is one part of our classification. 

Transcendental aspects of solutions of Painlev\'e equations \emph{over $\m C(t)$} were understood even before the work of the Japanese school. These results are most cleanly stated in terms of isolation of types, a notion coming from model theory. Throughout the manuscript, \emph{by a generic Painlev\'e equation}, we mean one where the complex parameters which appear in the equation are algebraically independent over $\mathbb{Q}$.

\begin{fact} \label{isolationPainleve}
Let $\mc X$ denote one of the five Painlev\'e families $II$ through $VI$ and let $\mc X(\alpha)$ denote the Painlev\'e equation with parameter value $\alpha$ (in general a tuple of complex numbers). 

\begin{enumerate}
    \item For generic values of $\alpha$, the equation $\mc X(\alpha )$ isolates a type over $\m C(t)$.\footnote{The notion of a first order formula isolating a type comes from model theory, but we will explain its incarnation in algebraic differential equations. Given a differential field $K$, and a system of differential equations $X$ over $K$ and inequalities $Y$ over $K$, then we say that the systems of inequalities $X$ and $Y$ isolate a type if every solution to the system $X$ in any differential field containing $K$ which satisfies the inequalities $Y$ is a generic solution of the system $X$ over $K$.} That is, $\mc X(\alpha)$ has no order one differential subvarieties over $\m C(t)$ and no solutions in $\m C(t)^{alg}$.
    
    \item There is a discrete Zariski-dense subset of $\alpha $ such that $X(\alpha)$ has a unique order one subvariety defined over $\m C(t)$. In this case, the subvariety is isomorphic to a Riccati equation with no solutions in $\m C(t)^{alg}$. As such the equation of the subvariety isolates a type over $\m C(t)^{alg}$. On this discrete set, the equation $X(\alpha)$ together with the inequation of the subvariety isolates a type over $\m C(t)$. 
    
    \item There is a discrete Zariski-dense subset of $\alpha$ and some finite number $N$ such that $X(\alpha)$ for $\alpha $ in the set has exactly $N$ solutions in $\m C(t)^{alg}$. For $\alpha$ in this discrete set, the equation $X(\alpha)$ and the inequations excluding the algebraic solutions isolate a type over $\m C(t)$. 
    
   \item  When $\mc X$ is $P_{III}, \, P_V,$ or $P_{VI}$, there is a discrete Zariski-dense subset of $\alpha$ such that $X( \alpha)$ for $\alpha $ in the set has exactly $2$ solutions in $\m C(t)^{alg}$. For $\alpha$ in this discrete set, the equation $X(\alpha)$ and the inequations excluding the two algebraic solutions isolate a type over $\m C(t)$. 

\end{enumerate}
\end{fact}

For (1) and (2) arguments for some of the families can be found in \cite{GromakLaineShimomura}. The $P_{II}$ and $P_{IV}$ portions of (3) come from \cite{Murata}. Condition (4) implies condition (3), and (4) applies to $P_{III}, P_V,$ and $P_{VI}$. The $P_{III}$ portion of (4) can be found in \cite{murata1995classical}. The $P_V$ portion of (4) is explained in \cite{watanabe1995solutions}. The $P_{VI}$ portion of (4) comes from \cite{lisovyy2014algebraic}, which gives a complete classification of algebraic solutions of $P_{VI}$. The extensive work of Japanese school of differential algebra went farther, and gives a complete understanding of the transcendental aspects of any \emph{individual solution} to a Painlev\'e equation over differential fields $K$ extending $\m C(t)$; this was carried out in a series of papers \cite{watanabe1998birational, watanabe1995solutions, umemura1997solutions, umemura1998solutions, murata1995classical, ohyama2006studies, okamoto1987studies, okamoto1986studies, okamoto1987studies5} and the results will be explained and used in the specific sections of this paper devoted to each individual Painlev\'e equation. In Section \ref{newstrmin} we do not use the results of the Japanese school, but only the weaker (and classically known) Fact \ref{isolationPainleve}. 

The connection between these transcendence results and model theoretic notions was pointed out in \cite{nagloo2017algebraic}, where Umermura's condition $J$ was related to strong minimality. In Section \ref{newstrmin} we give a new proof of the strong minimality (or Umemura's condition J) of Painlev\'e equations from third, fifth and sixth families with generic coefficients. Our proof technique is novel and avoids the extensive complications of previous proofs \cite{watanabe1998birational, watanabe1995solutions, umemura1997solutions, umemura1998solutions, murata1995classical, ohyama2006studies, okamoto1987studies, okamoto1986studies, okamoto1987studies5} by applying a recent result of \cite{freitag2021any}. Following our new proof of strong minimality of these equations, we turn towards the main subject of this manuscript and the natural extension of the transcendence results of the Japanese school: the classification of algebraic relations between solutions of Painlev\'e equations.

By a classification of algebraic relations, we mean to make a complete explicit list of the (differential) algebraic relations between solutions of Painlev\'e equations, but intermediate qualitative information is also of interest. To obtain such a classification, one would need to: 

\begin{enumerate}

\item Classify algebraic relations between solutions of a fixed Painlev\'e equation. 

\item Classify algebraic relations between solutions of equations from the same family with different parameters. 

\item Classify algebraic relations between solutions of equations from different families. 
\end{enumerate} 

We make substantial progress in each of the three areas, which we summarize in detail in \ref{algindconj}. Item (1) was the driving force between a recent series of papers Nagloo and Pillay \cite{nagloo2014algebraic, nagloo2015geometric, nagloo2017algebraic,nagloo2019algebraic}. In this paper, we also establish new results of this type for the second and third Painlev\'e equations. We make substantial progress on item (2) in this paper - to our knowledge we give the first substantial progress towards this portion of the classification following the discovery of the B\"acklund transformations for each family. Nagloo \cite{nagloo2015transformations} made the first substantial progress on item (3), partially answering a question of Boalch by showing that most pairs of Painlev\'e equations from different families with generic parameters are orthogonal. We complete this work, providing a complete answer to the question of Boalch. In Section \ref{algindconj}, we give a detailed description of our results, which are most cleanly stated using some differential algebraic terminology, which we introduce next. 


\subsection{Model-theoretic and differential algebraic notions}
In establishing the transcendence results mentioned in the overview above, we will use the language, results, and point of view of Kolchin's \emph{Differential Algebraic Geometry} \cite{KolchinDAG} along with the model theory of differential fields \cite{MMP}. The point of view of these two perspectives is very similar, with a universal domain $\mc U$ of Kolchin being a saturated model of the theory of differentially closed fields of characteristic zero, $DCF_0$ in language $\{+ \cdot, \delta , 0, 1 \}$. We will assume without loss of generality that the constants of $\mc U$ are given by the complex numbers $\m C.$ 

Zero sets of finite systems of differential polynomial equations with coefficients in a differential field $F$ are called \emph{differential varieties} over $F$ and are closed sets in the \emph{Kolchin topology}. The complete theory $DCF_0$ has quantifier elimination, and so the definable sets in $\mc U^n$ are given by finite systems of differential polynomial equations and inequalities. That is, the definable sets in $\mc U^n$ are precisely the constructible sets in the Kolchin topology on $\mc U^n.$ We refer the reader to \cite{MMP} for an exposition of the above connections. 

When $K$ is a differential subfield of $\mc U$, and $a $ is a tuple of elements from $\mc U$, by $K \langle  a \rangle, $ we mean the differential field generated by $K$ together with $ a$, which is the same as the field generated by $K$ and the derivatives of all orders of the elements of $ a$. We write $b \in acl(K, a)$ when $b$ is in the algebraic closure of $K\langle  a \rangle,$ denoted $K\langle  a \rangle ^{alg}$. So, $b \in acl(K, a)$ means that over $K$, $b$ is algebraic over $a$ and its derivatives. 

The differential varieties and definable sets we consider in this paper will all be \emph{finite-dimensional}; a differential variety $V$ over $F$ is finite-dimensional if for any $a \in V$, the transcendence degree $td _F ( F \langle a \rangle )$ is finite. The maximal such transcendence degree of any $a \in V$ is called the \emph{order} of $V$. A solution $b \in V$ is called \emph{generic over $F$}\footnote{or simply generic if $F$ has been fixed.} if $td _F ( F \langle b \rangle )$ is equal to the order of $V$. We will sometimes refer to the differential field $F$ generated by the coefficients of the equations and inequalities defining $V$ as the field of definition of $V$. The majority of the differential varieties in this paper will be over $F$, a finitely generated subfield of $\m C(t).$ Sometimes we are particular about the field of definition, and we consider $F$-definable sets, however, if we simply write $V \subset \mc U^n$ is definable, it means that $V$ is definable over some differential subfield of $\mc U$. 

\begin{defn}
A definable set $V \subseteq \mc U$ is \emph{strongly minimal} if it is infinite and has no definable subsets $W$ such that $W$ and $V \setminus W$ are both infinite. 
\end{defn}

From quantifier elimination and the axioms for differentially closed fields, it is easy to see that a differential variety is strongly minimal if and only if it has no proper differential subvarieties of order at least one. If $V$ is a definable set of order $n$, then $V$ is strongly minimal if and only if 
\begin{itemize}
    \item $V$ can not be written as the union of disjoint definable set of order $n$ \emph{and}
    \item For any differential field $K$ extending the field of definition of $V$, and any $ a \in V$, either $a$ is a generic solution of $V$ or $a \in K^{alg}$. 
\end{itemize}

Strongly minimal sets play an important role in this work, and in the work of the Japanese school; strong minimality is equivalent to Umemura's Condition (J) - see \cite{nagloo2017algebraic}. Again, the next two definitions can be made more generally for definable sets (or more generally for types) in arbitrary theories, but we will define only the special case applicable to our setting. 

\begin{defn} \label{trivialdef} Let $V \subset \mc U^m$ be a strongly minimal definable set of order $n$. $V$ is called \emph{geometrically trivial} if over any differential field $F$ extending the field of definition of $V$, given any $a_1, \ldots , a_d \in V$, generic on $V$ over $F$, if $td_F ( F \langle a_1, \ldots , a_d \rangle ) < n \cdot d$, then already for some pair $a_i, a_j,$ of elements of $V$, we have $td _F ( F \langle a_i, a_j \rangle ) < 2 \cdot n.$ 

If $V$ has the stronger property that for distinct $a_1, \ldots , a_d \in V$, we have that $td_F ( F \langle a_1, \ldots , a_d \rangle ) = n \cdot d$, then we call $V$ \emph{strictly disintegrated}. 
\end{defn}

\begin{defn} \label{orthodef} Let $V \subset \mc U ^n$ and $W \subset \mc U^m$ be strongly minimal definable sets. We say that $V$ and $W$ are \emph{nonorthogonal} if there is a definable relation $R \subset V \times W$ such that the projection of $R$ to $V$ is cofinite, the projection of $R$ to $W$ is cofinite, and both of these projection maps is finite-to-one. Otherwise, we say that $V$ and $W$ are orthogonal. 
\end{defn} 

Sometimes we refer to the relation $R$ in the previous definition as a \emph{definable finite-to-finite correspondence}. There is a non-obvious subtle issue at play in the definition. Suppose that $V, W$ are defined over $F$. A relation $R$, should it exist, might only be defined over a proper differential field extension $K \supset F$. In this situation, we say $V$ and $W$ are \emph{weakly orthogonal over $F$}, while $V$ and $W$ are \emph{not weakly orthogonal over $K$.}

 \subsection{Previous results and a detailed summary} \label{algindconj} 
Understanding the algebraic relations between solutions of a fixed Painlev\'e equation with generic coefficients was driving force behind the recent series of papers of Nagloo and Pillay \cite{nagloo2017algebraic, nagloo2014algebraic, nagloo2015geometric}, which we now describe. 

In \cite{nagloo2017algebraic} Nagloo and Pillay conjectured that solutions $y_1, \ldots , y_n$ to a \emph{fixed} Painlev\'e equation whose complex parameters are generic satisfy: \[\text{td} (\m C (t) (y_1, y_1 ' ,\ldots , y_n , y_n') / \m C(t)) = 2n,\] that is, \emph{solutions of a generic Painlev\'e equation together with their derivatives are algebraically independent}.\footnote{Here \emph{generic} means that the tuple of parameters is transcendental and algebraically independent over $\m Q$. So, the results of Pillay and Nagloo apply to $P_{II} ( \pi)$ or $P_{IV} (\pi, e^ \pi)$ but not to $P_{IV} (\pi , 3 \pi )$ or $P_{IV} ( \sqrt {2}, \pi).$}
We will refer to this and some related generalizations we give as the {\bf \emph{algebraic independence conjecture}.} In terms of the previous subsection, the conjecture claims that any Painlev\'e equation with generic coefficients is strongly minimal and strictly disentegrated \ref{trivialdef}. We mention that claims of this nature have been made in the past by Drach and Vesiot, and that the conjecture was established for the first Painlev\'e equation by Nishioka \cite{nishioka2004algebraic}. 

In \cite{nagloo2017algebraic}, Nagloo and Pillay proved a weak version of the conjecture; if equality does not hold in the above equation, then it must be that there are two solutions of the given equation $y_{i_1}, y_{i_2}$ such that \[ \text{td} (\m C (t) (y_{i_1}, y_{i_1} ' ,y_{i_2} , y_{i_2}') / \m C(t)) < 4.\] This weak version of the conjecture follows from establishing the \emph{strong minimality} and \emph{geometric triviality} of the given equation, which Nagloo and Pillay accomplish for all Painlev\'e equations with generic complex parameters \cite{nagloo2017algebraic}. In this paper, we give a fundamentally new proof of these facts for $P_{III}$, $P_V$, and $P_{VI}$. We also establish this conjecture for the $D_7$ family of $P_{III}$, a new result, as this family had not been considered by Nagloo and Pillay. 

In \cite{nagloo2014algebraic}, Nagloo and Pillay go farther, and establish the strong form of their conjecture for each of the families $P_{II}, P_{IV}, P_{V}$ (and prove a weaker statement for $P_{III}$ and $P_{VI}$), again all under the assumption that the complex parameters are generic. Recently, Nagloo \cite{nagloo2019algebraic} strengthened this result to establish the strong form of the conjecture for $P_{III}$ and $P_{VI}.$ In \cite{nagloo2015geometric}, Nagloo establishes the weak form of the conjecture for solutions of $P_{II} ( \alpha)$ as long as $ \alpha \notin \frac{1}{2}+\m Z$. 

Our rank calculations in the case of the solution set of $P_{II}$, which we call $X_{II} ( \alpha) $, show that while $X_{II} (\alpha)$ is not strongly minimal for $ \alpha \in \frac{1}{2} + \m Z$, the equation does have \emph{Morley rank one}.\footnote{A definable set of Morley rank one is a finite union of strongly minimal sets. See \cite{MMP} for details in the case of differential fields.} Specifically, for such $ \alpha$, $X_{II}( \alpha)$ has one order one differential subvariety, which we will denote by $R(\alpha)$. Our analysis allows for the generalization of the arguments of \cite{nagloo2015geometric} proving the geometric triviality of the strongly minimal set $X_{II}(\alpha) \setminus R( \alpha)$ for $\alpha \in \frac{1}{2} + \m Z$. That is, we have established the weak form of the algebraic independence conjecture for \emph{generic solutions over $\m C(t)$ of any of the equations in the second Painlev\'e family}. Note that solutions of a Painlev\'e equation with generic parameters (over $\m Q$) are \emph{automatically} generic solutions over $\m C (t)$ as the equation isolates the generic type, see Fact \ref{isolationPainleve}. So the algebraic independence conjecture for \emph{generic solutions} of arbitrary Painlev\'e equations is more general than the algebraic independence conjecture solutions of Painlev\'e equations with \emph{generic parameters}. 

We also establish in Section \ref{DegeneP3} a strong form of the conjecture for the $D_7$ family of the third Painlev\'e equation, a family in which we know of no previous results toward classifying the algebraic relations between solutions of a fixed equation. More generally, we establish the model theoretic ranks of each Painlev\'e equation from each family (each equation has Morley rank one) including not previously considered families of $P_{III}$ type. Our rank calculations open up the possibility of proving analogous statements for each of the Painlev\'e families, since we show that the generic type of any solution of any of the Painlev\'e equations in any of the families is rank one. 

To summarize, the weak form of the algebraic independence conjecture is known only for generic parameter value equations from the families (this follows from the present paper combined with \cite{nagloo2017algebraic, nagloo2017algebraic}) and any fiber of $P_{II}$ (combining this paper with \cite{nagloo2015geometric}). We establish a weaker condition, \emph{Morley rank one}, for any fiber of any of the families of equations by using various existing results of the Japanese school on irreducibility. We also give a fundamentally new proof of strong minimality for generic $P_{III},$ $P_V$ and $P_{VI}$. 

Finally, we formulate the following strong form of the algebraic independence conjecture for the generic solutions of any Painlev\'e equation: 
\begin{ques} \label{nagloopillay} For which values of the parameters in the Painlev\'e families is it true that if $y_1,  \ldots , y_n$ are generic solutions of that fixed Painlev\'e equation, then \[ \text{td} (\m C (t) (y_1, y_1 ' ,\ldots , y_n , y_n') / \m C(t)) = 2n? \]
\end{ques} 
Without stipulating that the solutions are generic, the question can not have an affirmative answer, since the above equality also does not hold for values of the parameters of equations in the $P_{VI}$-family such that the equation is not geometrically trivial (see Section 6 of the present paper); the fibers of $P_{VI}$ which are not geometrically trivial are nonorthogonal to Manin kernels of elliptic curves and are in one orbit of the group of B\"acklund transformations. In that case, there are algebraic relations among the solutions of the equation coming from torsion packets of elliptic curves. In the case of $P_{II} (\alpha )$ with $\alpha \in \m Z$ there is a so-called auto-B\"acklund transformation giving rise to nontrivial algebraic relations amongst solutions of $P_{II} (\alpha ).$

{\bf Algebraic relations between solutions of equations from the same family with different parameter values.} A notable property of the Painlev\'e families is the presence of bijective differential algebraic maps between the solution sets of the various equations in a fixed family, known as B{\"a}cklund transformations. The existence of these maps plays a key role in this paper. To our knowledge, ours is the first work which aims to show that these classically known transformations are the only algebraic relations between solutions of Painlev\'e equations from the same family with different paramter values. Specifically, we make considerable progress towards an affirmative answer to the following question: 

\begin{ques} \label{np1} For each of the Painlev\'e families, do the groups of affine transformations (B\"acklund transformations) give all instances of nonorthogonality of between distinct fibers of the family of equations? 
\end{ques} 

Our proof techniques here rely on our earlier rank analysis, use of the transformation groups, understanding the families of algebraic solutions, and some number theoretic arguments (using o-minimality and the Hilbert irreducibility theorem with some model theoretic techniques). 

{\bf Algebraic relations between solutions of equations from different families.}
The central question, regarding algebraic relations between Painlev\'e equations from different families is:
\begin{ques} \label{np2} Are generic solutions to equations from distinct Painlev\'e families orthogonal? 
\end{ques} 
After explaining the existing work along these lines, we summarize our results. 

In \cite{nagloo2015transformations}, Nagloo showed that Painlev\'e equations with generic parameters from \emph{many} pairs of the different families (I through VI) are orthogonal, implying that there are no B{\"a}cklund transformations between generic Painlev\'e equations from different families, partially answering a question of P. Boalch. Several of the pairs of families eluded the techniques Nagloo used in \cite{nagloo2015transformations}. 

Specifically, the question of whether generic $P_{II}$ is orthogonal to generic $P_{IV}$ and whether generic $P_{III} $ is orthogonal to generic $P_{V}$ were left open by \cite{nagloo2015transformations}. We answer both of these questions affirmatively, which gives a complete answer to the question of Boalch. So, combining our results with those of Nagloo, we have established: 

\begin{thm} \label{Naggy} 
	Any two Painlev\'e equations which have generic parameters and come from distinct families (I-VI) are orthogonal. 
\end{thm}

The proof of the theorem will be completed in Propositions \ref{Naggy1} and \ref{Naggy2} below. Nagloo \cite{nagloo2015transformations} discusses B{\"a}cklund transformations and orthogonality in detail, explaining the relations between the ideas and their significance in the theory of Painlev\'e equations, and we refer the reader there for a complete account. But briefly, orthogonality of two Painlev\'e equations implies that there can be no algebraic relations between their solutions (even over any differential field $K$ extending $\m C(t)$). 
 
Finally, we  give a variant which subsumes each of Questions \ref{nagloopillay}, \ref{np1}, \ref{np2}:
\begin{ques} \label{nagloopillay1} Let $y_1,  \ldots , y_n$ be generic solutions of Painlev\'e equations such that the following hold:

\begin{itemize} 
\item They are solutions coming from different families or have the property that they are not related by elements of the groups of B{\"a}cklund transformations associated with each equation.

\item The solutions are not from the fibers of $P_{VI}$ corresponding to Manin Kernels of elliptic curves.\footnote{See Section \ref{P6sec}; the algebraic relations between such solutions are completely understood.} 
\end{itemize}
Then, does the following equation hold? \[\text{td} (\m C (t) (y_1, y_1 ' ,\ldots , y_n , y_n') / \m C(t)) = 2n.\]
\end{ques} 

A positive answer to Question \ref{nagloopillay1} would constitute the strongest possible classification of the algebraic relations between solutions of Painlev\'e equations. 

\subsection{Besides algebraic relations} 
Besides our results towards the classification of algebraic relations of Painlev\'e equations, we obtain a number of other interesting new results. 

In Section \ref{newstrmin} we give a new proof of the irreducibility (strong minimality) of $P_{III},$ $P_V,$ and $P_{VI}$ with generic coefficients. Our technique involves understanding the algebraic solutions of the equations in the family while applying a recent result of \cite{freitag2021any}. Our approach to this transcendence result is new, and seems completely different that the proofs for $P_{VI}$ of Watanabe \cite{watanabe1998birational} and the later proof of Cantat and Loray \cite{cantat2009dynamics}. 

In Section \ref{DegeneP3}, we examine the so-called degenerate family of the third Painlev\'e equation, proving new transcendence results for solutions of equations in this family. For instance, we prove that the generic fiber of this family is strongly minimal, geometrically trivial and $\omega$-categorical. We conjecture that for any solution $x$ of an equation in the degenerate $P_{III}$ family, $x$ is algebraically independent from all other solutions of the equation. We prove this statement for all but at most two other solutions. 

Our results on the degenerate $P_{III}$ family are analogous to those of Nagloo and Pillay for the other families of Painlev\'e equations, where the structure of classical solutions to equations in the family was utilized. The degenerate $P_{III}$ family has no classical solutions. Our proofs rely on a (new) technique using details of the Zilber trichotomy from \cite{sanchez2019isolated}. 

\section{A new proof of strong minimality} \label{newstrmin}
In this section, we give a new proof of the strong minimality of Painlev\'e equations from third, fifth and sixth families with generic coefficients. Though the results were previously known (complete citations are given in sections \ref{twotwo}, \ref{P5} and \ref{P6sec}), the existing proofs are difficult and involve extensive computations. 


\begin{thm} \label{strongpain}
For generic parameters, the third, fifth, and sixth Painlev\'e equations are strongly minimal. 
\end{thm}

The proof of Theorem \ref{strongpain} will use a series of lemmas, which apply more generally than the results of the theorem and are themselves of independent interest. In this section, we adopt the notation of \cite{freitag2021any} for our particular case; in particular, consider the following condition: 

\begin{itemize}
\item[$(\mathcal C_m)$]
For any $m$ distinct solutions $a_1, \ldots , a_m$ of $X(a)$, not in $\m Q (t,a)^{alg}$, the collection $$a_1, a'_1, a_2, a'_2, \ldots , a_m, a'_m$$
is algebraically independent over $\m Q (t,a).$
\end{itemize}

\begin{lem} \label{C3}
Suppose $p$ is a stationary nonalgebraic type in $\text{DCF}_0$. If $\mathcal C_3$ holds of realizations of $p$, then $p$ is minimal. If $\mathcal C_2$ holds of realizations of $p$, then $p$ is internal to the constants. 
\end{lem}

The preceeding lemma is proved in \cite[3.5 and 3.6]{freitag2021any}. Note that in \cite{freitag2021any} there are separate conditions $C_m$ and $D_m$, with the distinction being that $D_m$ applies to a complete type, while $C_m$ is written for equations. Here there is no difference since the generic type of the equations we study is isolated by Fact \ref{isolationPainleve}.

\begin{lem} \label{C2i}
The condition $\mc C_2$ holds for the generic fiber of any Painlev\'e equation over the field of definition of the equation\footnote{The field of definition is the extension of $\m Q(t)$ generated by the parameters of the equation - e.g. in the case of $P_{II}(\alpha)$, the field is $\m Q(t, \alpha)$.}. 
\end{lem}

\begin{proof} Let $\mc X$ denote the family of equations being considered and $X(a)$ denote an equation in the family with generic parameter $a$ (one of the Painlev\'e families with $a$ given by an appropriate generic tuple of complex numbers). Let $K = \m Q(t,a)$ denote the field of definition of $X(a)$. 

If $\mc C_2$ fails over $K$ for $X(a)$, then there is a polynomial with coefficients in $K$, $p_1(x,x',y,y')$ which vanishes on some distinct solutions $x\neq y$ of the equation in question, $X(a)$. Perhaps clearing the denominator of $a,$ we can assume that there is a polynomial $p(x,x',y,y',a)$ with coefficient in $\mathbb{Q}(t)$ with the same property. 

Note that by the assumption of genericity of the parameters, by Fact \ref{isolationPainleve}, $x$ and $y$ are generic solutions over $K$. Further, again using isolation, it follows that: 
\begin{equation} \label{PC2} \tag{$\mathfrak{C}_2$}
\text{for any $x \in X(a)$ there is $y \in X(a)$ such that $p(x,x',y,y',a)=0$.}
\end{equation}

Now, as condition \ref{PC2} is a first order property over $\m Q(t)$, it follows that it holds on a definable subset of the space of parameters of the family $\mc X$. Since the property holds for a generic equation in the family, $X(a)$, it holds for equations $X(b)$ for $b$ from a $\m Q(t)$-definable dense set of parameters - so this includes a Zariski-open subset of the parameter space. 

There are two possibilities for the relation given by $p(x,x',y,y',a)=0$: 
\begin{enumerate}
    \item The variable, $y'$ appears nontrivially in $p$ and so for each $x$ there are infinitely many $y$ such that $p(x,x',y,y',a)=0$. 
    \item The variable, $y'$ does not appear in $p$ and so for each $x$ there are finitely many $y$ such that $p(x,x',y,y',a)=0$. 
\end{enumerate} 

In Case (1), this further property holds on the Zariski open subset of the parameter space. This open set includes a point $b$ such that $X(b)$ has an algebraic solution, and by Fact \ref{isolationPainleve}, this holds on some fiber with (finitely) many algebraic solutions as by Fact \ref{isolationPainleve} each family has a dense set of such points $b$. In each case, the generic type over $\m C(t)$ of the fiber is isolated by the equation $X(b)$ and excluding the algebraic solutions. But now take $x$ to be one of the algebraic solutions. Then the relation $p(x,x',y,y',b)=0$ gives an order one subvariety of $X(b)$ defined over $\m C( t)$. But, this is impossible by Fact \ref{isolationPainleve}. 

In Case (2), we have a polynomial relation of the form $p(x,x',y,a)=0$, and by Fact \ref{isolationPainleve} each family $X(a)$ has a dense set of fibers with a \emph{unique} subvariety of order one which is nonorthogonal over $\m Q(t)$ to a Riccati equation with no algebraic solutions.\footnote{In each case, describe this locus...} Taking $x$ to be a solution to such a Riccati equation in fiber $X(b)$, and $y$ algebraic over $x$ via $p$, the order $y$ over $\m Q(t,b)$ is one, so $y$ must be a solution to the Riccati equation as well. But any three solutions to this Riccati subvariety are algebraically independent over $\m C(t)$, by \cite[Proposition 2.2]{nagloo2019algebraic} or \cite[Section 4.3]{freitag2021bounding}, contradicting the existence of such a polynomial $p$. 
\end{proof}

Now we are ready to prove Theorem \ref{strongpain}. 

\begin{proof} 
Let $K = \m Q(t,a)$ denote the field of definition in the statement of the Theorem and $\mc X$ denoting the family of equations (the third, fifth or sixth Painlev\'e family) being considered and $X(a)$ denoting an equation in the family with generic parameter $a$ (a tuple of complex numbers). 

If $X(a)$ is not strongly minimal, then $\mc C_3$ fails for $X(a)$ over $\m Q(t,a)$, so suppose, for a contradiction, that this is the case. Then there is a polynomial $p(x,x',y,y', z,z',a)$ over $\m Q(t)$ and three distinct solutions $(\alpha, \beta, \gamma)$ of $X(a)$ such that $p(\alpha, \alpha ',\beta ,\beta ', \gamma , \gamma',a) = 0 $. We consider two cases based on the form of $p$. Note that since in Lemma \ref{C2i} we proved that condition $\mc C_2$ holds for $X(a)$, by Fact \ref{isolationPainleve}, the $2$-type of a solution is isolated by the equation $X(a)$ - that is every pair of solutions to $X(a)$ is generic over $\m Q(t,a)$. 

\begin{enumerate}
    \item The variable, $z'$ appears nontrivially in $p$ and so for each $x,y$ in $X(a)$ there are infinitely many $z$ such that $p(x,x',y,y',z,z',a)=0$.
    \item \sloppy The variable, $z'$ does not appear in $p$ and so for each $x,y$ in $X(a)$ there are finitely (bounded by the degree) many $z$ such that $p(x,x',y,y',z,a)=0$. 
\end{enumerate} 

In case 1), the conditions on $p$, which are first order (note that the theory of differentially closed fields eliminates the quantifier $\exists ^\infty$), must hold on a Zariski-open subset of the parameters $a$. In each case of $P_{III}, \, P_V, \, P_{VI}$, there is a dense set of parameters $a$ such that $X(a)$ has at least two algebraic solutions. Specializing to one of these fibers, $X(b)$ and letting $x,y$ be algebraic solutions of $X(b)$, it must be that $p$ gives an order one subvariety of $X(b)$ over $\m C(t)$, which contradicts Fact \ref{isolationPainleve} (none of the fibers which have algebraic solutions have order one subvarieties). 

In case 2), again, the conditions on $p$ are first order, and thus hold on a Zariski-open subset of the parameters $a$. Specialize to a parameter value $b$ in the Zariski-open subset for which $X(b)$ has exactly two algebraic solutions. There is a dense set of such fibers with exactly two algebraic solutions in each of the cases of $\mc X = P_{III}, \, P_V, \, P_{VI}$. But now letting $x,y$ be algebraic solutions of $X(b)$, any $z$ satisfying condition $2)$ must be an additional distinct algebraic solution, a contradiction. 
\end{proof}

\begin{ques}
The proof of Theorem \ref{strongpain} uses the existence of fibers of the family $\mc X$ which have exactly two algebraic solutions. This strategy does not work for $P_{II} $ or $P_{IV}$, where the fibers with algebraic solutions have unique algebraic solutions. Is there some similar strategy for establishing the irreducibility of $P_{II}$ and $P_{IV}$ with generic parameters? 
\end{ques}

\section{Painlev\'e two} \label{oneone} 
Much of the differential algebraic information in this section comes from \cite{umemura1997solutions}, whose notation we also follow. We also use the classification of algebraic solutions of equations of the second Painlev\'e family \cite{Murata}. The second Painlev\'e family of differential equations is given by 
\[ P_{II} (\alpha) : \, \, y'' = 2y^3 +ty + \alpha \] where $\alpha$ ranges over the constants. We will consider the following equivalent form of the equation given in Section 2.1 of \cite{umemura1997solutions}: 

\begin{align}
\label{p2} \tag{$X_{II} (\alpha)$} \begin{split} 
\frac{dq}{dt} & =  p-q^2 -\frac{t}{2} \\
\frac{dp}{dt} & =  2pq + \alpha + \frac{1}{2}.
\end{split}
\end{align}

We denote the set of solutions of the previous system by $X_{II} (\alpha )$. It is not hard to eliminate $p$ from the system to see that $$ \frac{d^2 q }{dt^2}  = 2q^3 +tq + \alpha,$$ so $y=q$ is a solution to $P_{II}$. Murata \cite{Murata} shows that $X_{II} ( \alpha)$ has a solution $(p,q)$ in $\m C(t)^{alg}$ if and only if $\alpha \in \m Z$; in this case, the solution is unique and $p,q \in \m C(t).$

For $\alpha \notin \frac{1}{2}+ \m Z$, by Theorem 2.1 of \cite{umemura1997solutions}, $X_{II} (\alpha )$ is strongly minimal. For $\alpha=-\frac{1}{2}$, Umemura and Watanabe \cite{umemura1997solutions}, see 2.7-2.9 on pages 169--170, show that if $K_1$ is a differential field extension of $K$ and if $q_1$ is a solution to $X_ {II} ( -\frac{1}{2}) $ such that the transcendence degree of $K_1 \langle q_1 \rangle $ over $K_1$ is one, then $q_1$ satisfies the Ricatti equation: \begin{align*}
      q_1 ' = q_1 ^2 +\frac{1}{2} t. \end{align*}
In model theoretic terms, this implies that  \[ \left\{q \, | \, q'' = 2q^3 +tq -\frac{1}{2} , \,  q' \neq q^2 + \frac{1}{2} t \right\} \] is strongly minimal, while $X_{II} ( -\frac{1}{2})$ is of Morley rank one and Morley degree two. 

Consider the group $G$ of affine transformations of $\m C$ generated by 

\begin{align*}
    \alpha & \mapsto -1-\alpha \\
    \alpha & \mapsto \alpha +1 \\
    \alpha & \mapsto \alpha -1
\end{align*}

\sloppy For every $g \in G$, there is a bijective rational map $g_* : X_{II} (\alpha ) \rightarrow X_{II} (g(\alpha ))$ \cite[see Section 2 for specific formulas]{umemura1997solutions}. We will sometimes refer to the maps $g_*$ as B\"acklund transformations.  Then if we denote by $X_{II} = \{ (p,q,\alpha) \, | \, (p,q) \in X_{II}(\alpha), \, \alpha \in \m C \},$ we have the following commutative diagram:

\[ \begin{CD}
X_{II} @>g_*>> X_{II}\\
@VV \pi V @VV \pi V\\
\m C @> g  >> \m C
\end{CD}\] 

where $\pi $ is the projection map to the $\alpha$-coordinate.

The differential varieties $X_{II} ( \alpha)$ for $\alpha \in \frac{1}{2} + \m Z$ are thus all isomorphic via B{\"a}cklund transformations, so our above analysis of $X_{II}(-\frac{1}{2})$ applies to $X_{II} ( \alpha )$ for $\alpha \in \frac{1}{2} +\m Z $. We denote by $R(\alpha )$ the Ricatti subvariety of $X_{II} (\alpha )$ for $\alpha \in \frac{1}{2}+ \m Z.$ Note that the degree of $R(\alpha )$ depends on $\alpha \in \frac{1}{2}+ \m Z.$ This can be seen by direct calculations, though the fact that the degree of the exceptional subvarieties can not be bounded uniformly over all $\alpha \in \frac{1}{2} +\m Z $ also follows by a compactness argument and the fact that for a generic coefficient $\alpha, $ $X_{II} ( \alpha)$ is strongly minimal. 

\begin{rem}
Our analysis in the preceding paragraphs contradicts the remarks in subsection 3.7 of \cite{nagloo2017algebraic}, where it is claimed that the Morley rank (and Lascar rank) of $X_{II}(\alpha)$ for $\alpha \in \frac{1}{2} + \m Z$ is two (in particular Fact 3.22 of \cite{nagloo2017algebraic} is incorrect). Parts of the subsequent discussion of that paper depend on this fact, but the main theorems of that paper and the following papers in the series do not. Particularly, this leaves Question 2.9 of \cite{hrushovski1999lascar} still open for order $3$ and higher; Hrushovski and Scanlon constructed a differential variety in which Lascar rank and Morley rank differ - it is not known if there are such examples of lower order than their construction. Freitag and Moosa \cite{freitag2017finiteness} show that Lascar rank and Morley rank are equal for order two differential varieties. The claim of \cite{nagloo2017algebraic} regarding the ranks of the second Painlev\'e family would have given an example of an order three equation with Lascar rank not equal to Morley rank, and so the question of Hrushovski and Scanlon would have then been definitively settled. 
The second Painlev\'e family witnesses the non-definability of Morley degree, rather than Morley rank as had been claimed in \cite{nagloo2017algebraic}. In the coming subsections, we will show that this is the case for each of the Painlev\'e families. 
\end{rem}

Next, we generalize a result of Nagloo \cite{nagloo2015geometric}: 

\begin{prop} The definable set 
\[ X= \left\{ y \, \vert \, y'' = 2y^3 +ty -\frac{1}{2} , \,  y' \neq y^2 + \frac{1}{2} t \right\}= X_{II} \left( -\frac{1}{2} \right) \setminus R(-\frac{1}{2} ) \] 
is strongly minimal and geometrically trivial.
\end{prop} 
\begin{proof} As established in the paragraphs above, the strong minimality of this set is a reinterpretation of the results of Umemura and Watanabe \cite{umemura1997solutions}, pages 169--170. With this in place, we will establish the triviality of this definable set via the argument of Nagloo \cite{nagloo2015geometric}, Proposition 3.3.\footnote{This argument depends on strong minimality of the definable set, so it would not have been accessible to Pillay and Nagloo, since their claim was that $X_{II}(\alpha) $ has Morley rank two for $\alpha \in \frac{1}{2}+\m Z$.} 
	
By the strong minimality of the above set $X$, the type of a generic solution to $P_{II} ( -\frac{1}{2})$ is of Morley rank one. The equivalence relation of nonorthogonality refines transcendence degree, so the type of a generic solution to $P_{II} ( -\frac{1}{2})$ is orthogonal to the constants. By a result of Hrushovski and Sokolovic \cite{HrSo},\footnote{For a published reference see \cite[Fact 4.1 part (a)]{sanchez2019isolated}.} any locally modular nontrivial strongly minimal set in differentially closed fields is nonorthogonal to the Manin kernel of a simple abelian variety. From this, it follows that a strongly minimal set $X$ is trivial if \emph{for any generic $x,y \in X$, if $y \in K \langle x \rangle ^{alg}$, then $y \in K \langle x \rangle $} (see \cite{nagloo2015geometric}, Proposition 2.7, for a proof of this fact). The remaining portion of the proof follows \cite{nagloo2015geometric}, Proposition 3.3; it can be verified that the strong minimality of $X_{II}(\alpha)$ is used in only one place in the proof of Proposition 3.3 of \cite{nagloo2015geometric}. Namely, in Claim 1 of the proof of Proposition 3.3, strong minimality is only used to show that the polynomial $F$ (defined in \cite{nagloo2015geometric}, Claim 1 in the proof of Proposition 3.3) cannot divide its derivative. The same applies in our case by strong minimality of $X$. The rest of the argument proceeds identically to \cite{nagloo2015geometric}, Proposition 3.3. 
	\end{proof} 
Because the B{\"a}cklund transformations give definable bijections between the sets $X_{II} (\alpha)$ for $\alpha \in \frac{1}{2} + \m Z$, it is the case that for each such fiber, the generic type of each fiber of the the second Painlev\'e equation is geometrically trivial. Thus, by combining our results with those of \cite{nagloo2015geometric}, we have established the weak form of the algebraic independence conjecture of \ref{algindconj} for all equations in the second Painlev\'e family. 

\subsection{Orthogonality of fibers.} Establishing the orthogonality of two strongly minimal sets defined over a parameter set $A$ requires, in general, one to consider extensions of the parameter set $A$. However, in the case of geometrically trivial strongly minimal sets, we have the following result, which follows from general results of geometric stability theory \cite{GST}, Corollary 5.5, chapter 2, section 5:

\begin{lem} \label{GST}  Let $X$ and $Y$ be nonorthogonal geometrically trivial (or even modular) strongly minimal definable sets over a differential field $K$. Then $X$ and $Y$ are non weakly orthogonal over $K$.  
\end{lem}

Recall, two strongly minimal sets $X$ and $Y$ are nonorthogonal if there is a generically finite-to-finite $K_1$-definable correspondence between $X$ and $Y$ for some $K_1$ extending $K$. We say $X$ and $Y$ are \emph{non weakly orthogonal} over $K$ if they are nonorthogonal and the correspondence can be taken to be definable over $K$. 

\begin{lem} \label{kernel} \begin{enumerate} 
\item The definable set $X_{II}(a) \setminus R(a)$ for $a \in \frac{1}{2}+ \m Z$ is orthogonal to $X_{II}(\alpha)$ where $ \alpha \in \m C $ is transcendental. 
\item The definable set $X_{II}(n)$ for $n \in \m Z$ is orthogonal to $X_{II}(\alpha)$ where $ \alpha \in \m C $ is transcendental. 
\end{enumerate} 
\end{lem} 
\begin{proof} If not, then by Lemma \ref{GST} and applying a suitable B\"acklund transformation there is a formula $\phi (\alpha, x,y)$ over $\m Q (t)$, which gives a generically  finite-to-finite correspondence between $X_{II}(\frac{1}{2}) \setminus R( \frac{1}{2})$ and $X_{II}(\alpha)$. Since $X_{II}(\alpha)$ has no $\m Q(t)^{alg}$-points and is strictly disintegrated, it must be the case that this correspondence is a finite-to-one surjective map from $X_{II}(\frac{1}{2}) \setminus R( \frac{1}{2})$ to $X_{II}(\alpha)$. The collection of $\alpha \in \m C$ such that $\phi (\alpha, x,y)$ is a finite-to-one surjective map is a $\m Q (t)$-definable subset of $\alpha \in \m C$. Strong minimality of the constants, plus the genericity of $\alpha$ implies that $\phi (\alpha, x,y)$ is finite-to-one surjective map from $X_{II}(\frac{1}{2}) \setminus R( \frac{1}{2})$ to $X_{II} ( \alpha) $ for all but finitely many $\alpha \in \m C$. 
Then for some $a \in \m Z$, $\phi (a, x,y)$ is a finite-to-one surjective from $X_{II}(\frac{1}{2}) \setminus R( \frac{1}{2})$ to $X_{II} ( \alpha) $. 
But, $X_{II} ( \alpha) $ has a solution in $\m C (t) ^{alg}$ if and only if $\alpha \in \m Z$ (cf. \cite{Murata}). Because the correspondence between $X_{II}(\frac{1}{2}) \setminus R( \frac{1}{2})$ and $X_{II} ( a) $ is defined over $\m Q(t) $ and is finite-to-one, we must have some element of $\m C ( t) ^{alg}$ in $X_{II}(\frac{1}{2}) \setminus R( \frac{1}{2})$, contradicting Murata's results \cite{Murata}. 

For the second statement, the proof is very similar to the first. On a cofinite subset of $\alpha \in \m C$, we have a $\m Q (t)$-formula $\phi (\alpha, x,y)$ which defines a finite-to-one surjective map from $X_{II}(n)$ to $X_{II} ( \alpha) $. The domain of the map is all but a fixed finite subset of $X_{II}(n)$ for almost all $\alpha$; as this set is definable over $\m Q(t) \subset \m C (t)^{alg}$, it must be that the domain is either all of $X_{II}(n)$ or is all but the unique algebraic solution of $X_{II}(n)$. Now we argue the two cases separately. If this domain includes the algebraic solution given by \cite{Murata}, then we have a contradiction as above for any $\alpha \notin \m Z$. If this domain does not include the algebraic solution of $X_{II}(n)$, then we obtain a contradiction by obtaining a map (via specializing $\alpha$) from $X_{II}(n)$ to $X_{II}(m)$ for some other $m \in \m Z$ which is surjective and whose domain does not include the algebraic solution of $X_{II}(n)$. The range does include an algebraic solution, and the map is defined over $\m Q (t)$, a contradiction. 
\end{proof} 

\begin{prop} \label{kernel1} Let $a \in \m Q^{alg}$ and let $\alpha$ be transcendental.  Then $X_{II} (a)$ is orthogonal to $X_{II} (\alpha)$. 
\end{prop} 
\begin{proof} 
Assume that $a \notin \frac{1}{2} + \m Z$ and $a \notin \m Z$, since these cases are covered by Lemma \ref{kernel}. By the same argument of Lemma \ref{kernel}, we obtain a $\m Q(a,b,t)$-definable finite-to-one surjective map from $X_{II} (a)$ to $X_{II} (b)$ for some $b \in \m Z$. Now using Murata's theorem, we obtain a contradiction, since there is an element of $\m C (t) ^{alg}$ in $ X_{II} (b)$ but not $X_{II} (a)$.
\end{proof}


From the previous two lemmas, an affirmative answer to the following question seems plausible: 
\begin{ques} \label{ortho2} For any $\alpha, \beta \in \m C$, consider the following conditions:
\begin{itemize} 
\item $X_{II} ( \alpha )$ is nonorthogonal to $X_{II} (\beta)$
\item $\beta - \alpha \in \m Z$ or $\beta + \alpha \in \m Z$.
\end{itemize} 
Are the two conditions equivalent? 
\end{ques} 

We answer the question \emph{generically:} 

\begin{prop} \label{genone} Suppose that $\alpha, \beta \in \m C,$ with at least one of $\alpha, \beta $ transcendental. Then $X_{II} ( \alpha )$ is nonorthogonal to $X_{II} (\beta)$ if and only if $\beta - \alpha \in \m Z$ or $\beta + \alpha \in \m Z$. In the case that $\alpha, \beta $ are transcendental and $\beta - \alpha \in \m Z$ or $\beta + \alpha \in \m Z$, the only algebraic relations between solutions of $P_{II} ( \alpha)$ and $P_{II} ( \beta )$ are induced by the B\"acklund transformation $X_{II} ( \alpha ) \rightarrow X_{II} ( \beta )$. 
\end{prop} 
\begin{proof} By Proposition \ref{kernel1} we need only consider the case that both $ \alpha $ and $\beta$ are transcendental. If $\alpha$ and $\beta$ are interalgebraic by some polynomial $p(x,y) \in \m Q (x,y)$ (which we may assume to be irreducible), we can assume for cofinitely many $\alpha_1 \in \m C$, for any $\beta _1$ such that $p(\alpha_1, \beta _1)=0$, we have that $X_{II} ( \alpha_1)$ and $X_{II} (\beta_1)$ are nonorthogonal by a particular fixed formula $\phi (x,y, \alpha_1, \beta _1)$ which can be assumed be a $\m Q (t, \alpha_1, \beta _1)$ formula, by Lemma \ref{GST}. In particular, for all but finitely many values $\alpha_1 \in \m Z$, we have that $X_{II} ( \alpha_1)$ and $X_{II} (\beta_1)$ are in finite-to-finite correspondence, by $\phi (x,y, \alpha_1, \beta _1)$; in fact, since we know that $X_{II}$ is strictly disintegrated for $\alpha_1$ generic \cite{nagloo2014algebraic}, it must be that $\phi (x,y, \alpha_1, \beta _1)$ defines a bijection between $X_{II} ( \alpha_1)$ and $X_{II} (\beta_1)$ for all but finitely many $\alpha_1$. 

In particular, this holds for cofinitely many $\alpha_1 \in \m Z$. For $\alpha_1 \in \m Z$, it must be that $\beta_1 \in \m Z$ (again, this uses Murata's analysis of algebraic solutions \cite{Murata}). Further, for any integer $\alpha_1,$ the collection of $\beta \in \m C$ such that $p ( \alpha _1 , \beta )=0$ must consist entirely of integers $\beta$. Similarly, for any integer $\beta_1,$ the collection of $\alpha \in \m C$ such that $p ( \alpha  , \beta_1 )=0$ must consist entirely of integers $\alpha$. By the Hilbert Irreducibility theorem, \cite{lang1983fundamentals} Chapter 9, for all but finitely many integers $\alpha_1$, the polynomial $p(\alpha_1 , y)$ is irreducible over $\m Q$. Consider the coefficient, $p_1(x) \in \m Q [x]$, of the highest degree term of $p(x,y)$ considered as a polynomial in $y$. For all but finitely many integers $\alpha_1$, $p_1 ( \alpha_1) \neq 0$. But, now there must be an integer $\alpha_1$ such that $p(\alpha_1, y)$ has only integer roots and $p ( \alpha _1 , y)$ is irreducible and the leading degree monomial of $p(x,y)$ in $y$ does not vanish at $x = \alpha_1$. This is impossible unless $p (x,y)$ is linear in $y$. By a symmetric argument, $p(x,y)$ is also linear in $x$. We claim that $p(x,y)$ is (up to multiplication by a scalar) of the form $y-x+c$ or $y+x+c$ for some $c \in \m Z$, in which case, we have established the proposition in the case that $\alpha$, $\beta$ are algebraically dependent.

To see that $p(x,y)$ is of one of the two specified forms, we need only prove that there is no monomial term of the form $xy$ in $p(x,y)$. Should such a term appear, then $p(x,y) = xy + ax+by + c$ (up to multiplication by a constant) - by irreducibility, one can assume that at least one of $b,c,$ are nonzero. Then $y= -\frac{ax+c}{x+b}$ and so $\frac{dy}{dx}$ is nonconstant. Thus, either $\frac{dy}{dx}$ or $\frac{dx}{dy}$ is bounded in absolute value by $1$ for sufficiently large $x$ or $y$, respectively. Either case produces a contradiction, since for all sufficiently large integers, $x$ (or $y$) the polynomial $p(x,y)$ has only integer solutions in $y$ (in $x$). 

If $\alpha$, $\beta$ are transcendental and independent over $\m Q$, then by similar arguments as above, we have some formula $\phi (x,y, \alpha, \beta )$ which is a bijection between $X_{II} ( \alpha)$ and  $X_{II} (\beta)$ on some $\m Q$-definable subset of $\m C^2.$ But, since this definable set includes the generic point $(\alpha, \beta)$, the set must include a Zariski open subset of $\m C^2$. Any such subset includes a pair of points $(\alpha_1, \beta_1)$ so that $\alpha _1 \in \m Z$ and $\beta _1 \notin \m Z$, a contradiction by Murata's classification of algebraic solutions \cite{Murata}. To see the final statement of the Proposition, note that if there were some algebraic relation not induced by the B\"acklund transformation, we would obtain (by composing with the inverse of the B\"acklund transformation) an algebraic relation between two solutions of $P_{II} ( \alpha), $ which contradicts the fact that the equation is strictly disintegrated.
\end{proof} 

There are likely many methods of showing that $p(x,y)$ must have a very restrictive form\footnote{For instance the curve $C$, which we can assume to be nonsingular, defined by $p(x,y)$ has infinitely many integral points, so by Siegel's theorem \cite{hindry2013diophantine}, Part D, it must be genus zero with fewer than three points at infinity in its projective closure.} in the previous proof besides applying the Hilbert Irreducibility theorem and our arguments. One can show that $p(x,y)$ has the appropriate form via an elementary argument using o-minimality of the field of real numbers - essentially we will repeatedly be using the fact that in an o-minimal structure, a definable property holds for all sufficiently large $x$ if it holds for arbitrarily large $x$ (for a reference on o-minimality, see \cite{van1998tame}). Since $p$ has coefficients in $\m Q$, the map $x \rightarrow (y_1(x), \ldots , y_n(x))$ where $p(x,y_i(x))=0$ is a map which is interpretable in the real numbers considered as an ordered field. For sufficiently large integers, $x,$ all of the roots of $p(x,y)=0$ are integers and for sufficiently large $y,$ all of the roots of $p(x,y)=0$ are integers. By o-minimality the same is true for sufficiently large real-valued $x$, the roots of $p(x,y)$ are real-valued (and likewise for large $y$).

For sufficiently large $x$ (or $y$), the number of distinct (real) roots of $p(x,y)=0$ in $y$ (resp. in $x$) is some fixed number, and letting $n$ be the number of such roots, each of the maps $x \mapsto y_i$ where $y_1 < y_2 \ldots < y_n$ is a definable map. For sufficiently large $x$, we may assume, each of these maps is nonconstant and differentiable with differentiable inverse. Now, if $\frac{dy_i}{dx}$ is identically $\pm 1$, then the map must be of the form $y = x +c$ or $y= -x +c$, and so the linear polynomial $y -x -c$ or $y +x -c$ is a factor of $p(x,y)$ and it must be the case that $c \in \m Z$. 

If not, then for sufficiently large values of $x$, it must be that for some $i$, $\frac{dy_i}{dx} \in I$ where $I$ is one of the intervals $(1, \infty)$, $(0,1)$, $(-1,0)$, $(-\infty , -1)$. In the case that $I=(0,1)$, selecting some sufficiently large $\alpha_1 \in \m Z$, it must be that $y_i (\alpha_1 +1)$ is not is $\m Z$, since $y_i ( \alpha_1) \in \m Z$ and the derivative of $y_i$ on the interval $[\alpha_1, \alpha_1+1]$ is in $(0,1)$. A similar argument applies when $I = (-1,0)$. When $I$ is one of the unbounded intervals, then for sufficiently large values of $y_i$, $\frac{dx}{dy_i}$ is one of the bounded intervals, $(0,1)$, $(-1,0)$, and the same argument applies with the roles of $x$ and $y$ exchanged. It follows that $p(x,y)$ must be of the form $y-x+c$ or $y+x+c$ for some $c \in \m Z$.

\section{Painlev\'e three} \label{twotwo} 
The third Painlev\'e family of equations is given by \begin{equation} \frac{d^2y}{dt^2} = \frac{1}{y} \left( \frac{dy}{dt} \right) ^2- \frac{1}{t} \frac{dy}{dt} + \frac{\alpha y^2 + \beta }{t}+ \gamma y^3+ \frac{\delta}{y}.  \label{P311} \end{equation} For purposes of this section, we will work with an alternate Hamiltonian form of the equations which can be obtained from equation (\ref{P311}) in the so-called non-degenerate case that $\gamma \delta \neq 0.$ The existing work on $P_{III}$ from the model theoretic perspective has been entirely in this non-degenerate case, and in this section we restrict ourselves to this case. The degenerate case will be considered in Section \ref{DegeneP3}.

The differential algebraic information in this section comes mainly from \cite{umemura1998solutions}, whose notation we also follow. By \cite{okamoto1987studies}, assuming $\gamma \delta \neq 0,$ equation (\ref{P311}) can be transformed to the following system in Hamiltonian form, which we denote by $S_{III}( \bar v)$:
\begin{align}
\label{P3D6}
\tag{$S_{III}$} 
\begin{split}
t q ' & =   2 q^2 p -q^2 -v_1 q +t  \\ 
t p' &=  -2 qp^2 + 2 qp -v_1 p + \frac{1}{2} (v_1 +v_2)
\end{split}
\end{align} 
The series of transformations producing $S_{III}(\bar v)$ is also given in some detail in the introduction of \cite{umemura1998solutions}. Note that $\bar v = (v_1, v_2) \in \m C^2$. We denote, by $X_{III} ( v_1, v_2)$, the solution set of $S_{III}( v_1, v_2)$.

Define \begin{align*} W_1 = \left\{ \bar v \in \m C^2 \, | \, v_1 -v_2 \in 2 \m Z \right\} \end{align*}  and \begin{align*} D_1 = \left\{ \bar v \in \m Z^2 \, | \, v_1 +v_2 \in 2 \m Z \right\}.\end{align*}  
Theorem 1.2 (iii) \cite{umemura1998solutions} implies that for $\bar v$ not in $W_1$ or $D_1$, $X_{III}(\bar v)$ is strongly minimal. For $\bar v \in W_1$, Lemma 3.1 of \cite{umemura1998solutions} implies that $X_{III}( \bar v)$ has Morley rank one and Morley degree two.  For $\bar v \in D_1$, Lemma 3.2 of \cite{umemura1998solutions} implies that $X_{III}( \bar v)$ has Morley rank one and Morley degree three. To see these latter two facts from the statements of Lemmas 3.1 and 3.2, we describe the group of B\"acklund transformations acting on $X_{III}$. 

Consider the group $G$ of affine transformations of $\m C^2$ generated by:\begin{align*}
     s_1 (\bar v)& = (v_2, v_1) \\
     s_2 (\bar v)& = (-v_2 , -v_1)\\
     s_3 (\bar v )& = (v_2+1, v_1-1)\\
     s_4 ( \bar v )& = (-v_2+1, -v_1+1)
     \end{align*}
     
     Then for any $g \in G$, we have the following commutative diagram: 
     
     \[ \begin{CD}
X_{III} @>g_*>> X_{III}\\
@VV \pi V @VV \pi V\\
\m C^2 @> g  >> \m C^2
\end{CD}\] 

Where $g_*$ is a rational bijective map (between the differential varieties $X_{III} (v_1, v_2)$ and $X_{III} (g(v_1, v_2))$). We now state a result of Murata \cite{murata1995classical} for convenience and because we are following the notation of \cite{umemura1998solutions} which differs slightly from the notation of \cite{murata1995classical}. 

\begin{thm} \label{Murata3} \cite{murata1995classical}, Section 1.3. The equation $S_{III} (v_1, v_2)$ has algebraic solutions if and only if $v_2-v_1-1 \in 2 \m Z$ or $v_2+v_1 +1 \in 2 \m Z$. If there are algebraic solutions to $S_{III} (v_1, v_2) $, then there are either two or four. There are four algebraic solutions precisely when both $v_2-v_1-1 \in 2 \m Z$ and $v_2+v_1 +1 \in 2 \m Z$.
\end{thm} 

\begin{prop} \label{orthofam3} For any $\alpha \in \m  C $, and any $\beta, \gamma \in \m C$ which are independent and transcendental over $\m Q$, $X_{II}(\alpha ) $ is orthogonal to $X_{III} ( \beta , \gamma)$. 
\end{prop} 
\begin{proof} First, let $\alpha \in \m Q^{alg} \setminus (\m Z \cup (\frac{1}{2} + \m Z))$. By Lemma \ref{GST}, the triviality and strong minimality of $X_{III} ( \beta , \gamma)$ \cite{nagloo2017algebraic} and $X_{II} (\alpha )$, and the fact that $X_{III} ( \beta , \gamma)$ has no algebraic solutions \cite{umemura1998solutions} implies that there must be a $\m Q(t, \alpha, \beta , \gamma)$-definable finite-to-one map from $X_{II}(\alpha )$ to $X_{III} ( \beta , \gamma)$. By the genericity of $\beta, \gamma$, there is such a map on a Zariski open subset of $(\beta_1, \gamma_1 ) \in \m C^2$. Such an open subset includes some points with $\gamma_1 - \beta_1 -1 \in 2 \m Z$, for which $X_{III} ( \beta_1 , \gamma_1)$ has an algebraic solution by Theorem \ref{Murata3}, but $X_{II} ( \alpha)$ has no algebraic solution, a contradiction. The same argument applies to $X_{II}(\alpha ) \setminus R(\alpha )$ when $ \alpha \in \frac{1}{2} + \m Z$. The set $R(\alpha)$ is order one and strongly minimal, and nonorthogonality refines algebraic dimension for strongly minimal sets, so $R( \alpha)$ is orthogonal to $X_{III} ( \beta , \gamma)$. 

Now, suppose that $\alpha \in \m Z$. Then just as above, we have a $\m Q(t, \beta , \gamma)$-definable map from $X_{II}(\alpha )$ to $X_{III} ( \beta , \gamma)$, and the domain of the map could either be all of $X_{II}(\alpha )$ or $X_{II}(\alpha )$ minus the algebraic solution. In the first case, we have a contradiction since $X_{III} ( \beta , \gamma)$ has no algebraic solutions. In the second case, we obtain a contradiction by specializing $\beta, \gamma$ to values $\beta_1 , \gamma_1$ such that $\gamma - \beta -1 \in 2 \m Z$. 

Finally, Proposition 5.4 of \cite{nagloo2015transformations} covers the case in which $\alpha$ is transcendental. 
\end{proof} 

The previous argument shows that simple specialization arguments yield the orthogonality some Painlev\'e equation with generic coefficients with a Painlev\'e equation having algebraic coefficients. Similar arguments can be easily adapted for all pairs of equations from all of the families. We will not make similar precise statements for other pairs of equations. Even the condition of genericity can generally be relaxed as long as at least one of the coefficients is transcendental: 

\begin{prop} Suppose that the transcendence degree of $(\beta, \gamma )$ is one and $\alpha \in \m Q^{alg}$. If $\alpha \notin \frac{1}{2} + \m Z$, then $X_{II}(\alpha ) $ is orthogonal to $X_{III} ( \beta , \gamma)$. If $\alpha \in \frac{1}{2}+ \m Z$, $X_{II}(\alpha ) \setminus R(\alpha )$ is orthogonal to $X_{III} ( \beta , \gamma)$.
\end{prop} 

\begin{proof} 
If (the generic component of) $X_{III} ( \beta , \gamma)$ is not geometrically trivial, then it must be the case that $X_{II} ( \alpha )$ (or $X_{II}(\alpha ) \setminus R(\alpha )$ in case $\alpha \in \frac{1}{2}+ \m Z$) is orthogonal to $X_{III} ( \beta, \gamma)$. Thus, we can assume that $X_{III} ( \beta , \gamma)$ is geometrically trivial. Then by Lemma \ref{GST}, we have a $\m Q(t)$-formula $\phi (x,y, \alpha,  \beta ,\gamma)$ which defines a generically finite-finite correspondence between $X_{II} ( \alpha )$ and $X_{III} ( \beta , \gamma)$ (with the variable $x$ corresponding to $X_{II} ( \alpha )$ and the variable $y$ corresponding to $X_{III} ( \beta , \gamma)$).  

{\bf Case I:} Suppose that $\alpha \in \m Q^{alg} \setminus ( \m Z \cup (\frac{1}{2} + \m Z)).$ Then the domain of $\phi (x,y, \alpha,  \beta ,\gamma)$ must be all of $X_{II}(\alpha)$ since $X_{II} ( \alpha)$ has no $\m C(t)^{alg}$-points. 

{\bf Case Ia:} Suppose that $\beta +\gamma =n \in 2 \m Z$. Then the range of the correspondence in $X_{III} ( \beta , \gamma) $ must be all of $X_{III} ( \beta , \gamma) $ except for the exceptional subvariety given by \cite{umemura1998solutions} (described at the beginning of this section). These properties of $\phi (x,y, \alpha, \beta , \gamma)$ hold for cofinitely many points in the locus of $(\beta , \gamma) $ such that $ \beta + \gamma =n \in 2 \m Z$, but the set of $(\beta, \gamma)$ such that $ \beta + \gamma =n \in 2 \m Z$ contains infinitely many points such that $ \beta - \gamma +1 \in 2 \m Z$. This contradicts the fact that $X_{II} ( \alpha)$ has no algebraic solutions and $X_{III} (\beta , \gamma)$ has two algebraic solutions. The case that $\beta -\gamma =n \in 2 \m Z$ is analogous. 

When $\beta -\gamma  \notin \m Z$ and $\beta + \gamma  \notin \m Z$, $X_{III} ( \beta , \gamma)$ is strongly minimal. We have two additional cases when $X_{III} ( \beta , \gamma)$ is strongly minimal. 

{\bf Case Ib:} Suppose $\beta -\gamma +1=n  \in 2 \m Z$. Since $X_{II} ( \alpha ) $ has no algebraic solutions, we must have that the range of $\phi (x,y, \alpha,  \beta ,\gamma)$ is all of $X_{III} ( \beta , \gamma)$ except for two points. This property holds on all but finitely many $(\beta , \gamma)$ such that $\beta -\gamma +1  =n$. But, now among the collection of $(\beta , \gamma)$ such that $\beta -\gamma +1 =n$, there are infinitely many $(\beta , \gamma)$ such that $\beta +\gamma +1 \in 2 \m Z$. For such $( \beta , \gamma)$, $X_{III} ( \beta , \gamma)$ has four algebraic solutions. Thus, two of these solutions are in the range of $\phi (x,y, \alpha,  \beta ,\gamma)$, contradicting the fact that $X_{II} ( \alpha)$ has no algebraic solutions. The case that $\beta +\gamma +1 \in 2 \m Z$  is analogous. 

{\bf Case Ic:} Suppose $\beta$ and $\gamma$ satisfy some algebraic relation other than Case Ib. Then we must have that the range of $\phi (x,y, \alpha,  \beta ,\gamma)$ is all of $X_{III} ( \beta , \gamma)$, and this holds on some specialization of $(\beta, \gamma)$, $(\beta_1, \gamma_1)$ such that $\beta_1 +\gamma_1 +1 \in 2 \m Z$ or $\beta_1 -\gamma_1 +1 \in 2 \m Z$, which gives a contradiction, since for such $(\beta_1, \gamma_1)$, $X_{III} (\beta_1, \gamma_1)$ has two algebraic solutions and $X_{II} ( \alpha )$ has no algebraic solutions. 

{\bf Case II:} The case in which $\alpha \in \frac{1}{2} + \m Z$ can be argued identically to Case I. 

{\bf Case III:} Suppose that $\alpha \in \m Z$. As in Case I, we obtain $\phi (x,y, \alpha,  \beta ,\gamma)$, a generically finite-to-finite correspondence between $X_{II} ( \alpha )$ and $X_{III} ( \beta , \gamma)$. If the domain of $\phi (x,y, \alpha,  \beta ,\gamma)$ does not include the unique algebraic point in $X_{II} (\alpha)$, then the argument proceeds identically as in Case I. So, suppose that the domain of $\phi (x,y, \alpha,  \beta ,\gamma)$ is all of $X_{II} ( \alpha )$. Then the range of $\phi (x,y, \alpha,  \beta ,\gamma)$ must contain some algebraic point(s) or a subvariety, which means that we must have one of the following relations holding: 
\begin{enumerate} 
\item $\beta -\gamma +1=n \in 2 \m Z.$
\item $\beta +\gamma +1=n  \in 2 \m Z.$
\item $\beta -\gamma =n \in 2 \m Z.$
\item $\beta +\gamma =n  \in 2 \m Z.$
\end{enumerate} 
Case i) is similar to Case ii), and Case iii) is similar to Case iv). Thus, we will cover cases i) and iii). Assume that $\beta -\gamma +1=n  \in 2 \m Z. $ Then on a Zariski open subset of $(\beta , \gamma )$ such that $\beta -\gamma +1=n $, $\phi (x,y, \alpha,  \beta ,\gamma)$ is a finite-to-finite correspondence so that one of the two following cases holds: 
\begin{itemize} 
\item Let $a \in X_{II} ( \alpha)$ be the unique algebraic solution. The range of $\phi (x,y, \alpha,  \beta ,\gamma)$ is all of $X_{III} (x,y,\alpha, \beta, \gamma)$ and $$| \{y \, : \, y \in X_{III} ( \beta , \gamma) \, \text{ and } \phi (a,y, \alpha,  \beta ,\gamma) \} | =2.$$
\item Let $a \in X_II ( \alpha)$ be the unique algebraic solution. The range of $\phi (x,y, \alpha,  \beta ,\gamma)$ is all of $X_{III} (x,y,\alpha, \beta, \gamma)$  except for one point and $$| \{y \, : \, y \in X_{III} ( \beta , \gamma) \, \text{ and } \phi (a,y, \alpha,  \beta ,\gamma) \} | =1.$$
\end{itemize} 
In either of these cases, picking $(\beta, \gamma)$ so that $\phi (x,y, \alpha,  \beta ,\gamma)$ has one of the above properties and $\beta +\gamma +1  \in 2 \m Z.$ Then we have a contradiction since $X_{III} ( \beta , \gamma)$ has four algebraic solutions. 

Now, assume we are in Case iii), that is $\beta -\gamma =n \in 2 \m Z.$ The set $X_{III} ( \beta , \gamma)$ has no algebraic solutions, so it must be that the fiber of $\phi (x,y, \alpha,  \beta ,\gamma)$ above the unique algebraic point of $X_{II} ( \alpha)$ must be the unique order one subvariety of $X_{III} ( \beta , \gamma)$. Then the range of $\phi (x,y, \alpha,  \beta ,\gamma)$ must be all of $X_{III} ( \beta , \gamma)$, and this holds on a cofinite subset of all $(\beta , \gamma)$ such that $\beta -\gamma =n .$ In particular, this holds for all but finitely many $(\beta , \gamma)$ such that $\beta + \gamma \in 2 \m Z$ and $\beta -\gamma =n .$ By the results of \cite{umemura1998solutions} as interpreted above, 
\begin{itemize} 
\item When $\beta + \gamma \in 2 \m Z$ and $\beta -\gamma =n ,$ there are two irreducible order one subvarieties of $X_{III}( \beta , \gamma )$. 
\item When $\beta + \gamma \notin 2 \m Z$ and $\beta -\gamma =n ,$ there is one irreducible order one subvariety of $X_{III}( \beta , \gamma )$. 
\end{itemize}
It follows that the degree of the differential polynomial defining these order one subvarieties must be unbounded, since otherwise there is a Zariski dense subset of $\{(\beta , \gamma) \, | \, \beta -\gamma =n \}$ which has two disjoint infinite definable order one subsets defined by the instances of differential polynomials of bounded degree. It would then follow that there are such a subvarieties for generic $(\beta, \gamma)$ such that $ \beta -\gamma =n.$ But, we know by \cite{umemura1998solutions} that there are not two such definable subsets for $(\beta, \gamma)$ such that $ \beta -\gamma =n.$ 

But, on a cofinite subset of $(\beta , \gamma)$ such that $ \beta -\gamma =n,$ the correspondence $\phi (x,y, \alpha,  \beta ,\gamma)$ is surjective onto $X_{III} ( \beta , \gamma)$. It follows that the fiber above the algebraic point of $X_{II} ( \alpha)$ must consist of the order one subvarieties of $X_{III} ( \beta , \gamma)$. But it is impossible for the these subvarieties to be the fibers of a single formula $\phi, $ since as we explained above, their degree is unbounded. 
\end{proof} 

\begin{ques} \label{nongen3orth}
For any complex parameters, $\alpha, \beta , \gamma,$ are the generic types of $X_{II}(\alpha )$ and $X_{III} ( \beta , \gamma)$ orthogonal?
\end{ques}

We now turn to analyzing the nonorthogonality classes within the third Painlev\'e family. 

\begin{prop} \label{P3orth} Suppose that $(v_1, v_2)$ are algebraically independent and transcendental and that $(w_1,w_2)$ are algebraically independent and transcendental. For $x \in \m C$, let $\pi (x)$ denote $x (\text{mod} \, 2 \m Z).$ Then $X_{III} ( v_1, v_2 ) $ is nonorthogonal to $X_{III} ( w_1, w_2)$ if and only if the sets $\{\pi (v_2-v_1) , \pi (v_1 -v_2)\}$ and $\{\pi (w_2-w_1) , \pi (w_1 -w_2) \}$ are identical. In this case, the only algebraic relations between solutions of the equations are induced by B\"acklund transformations. 
\end{prop} 
\begin{proof} 
If the sets $\{\pi (v_2-v_1) , \pi (v_1 -v_2)\}$ and $\{\pi (w_2-w_1) , \pi (w_1 -w_2) \}$ are identical, there is a composition of the maps given $s_i$ given above sending $(v_1, v_2)$ to $(w_1, w_2)$; the corresponding B{\"a}cklund transformation gives the nonorthogonality of the two sets. 

Now, assume that $X_{III} ( v_1, v_2 ) $ is nonorthogonal to $X_{III} ( w_1, w_2)$. Because of the strong minimality and triviality of $X_{III} (v_1, v_2)$ and $X_{III} (w_1, w_2)$ \cite{nagloo2017algebraic}, by Lemma \ref{GST}, there is a definable bijection $\phi (x,y, v_1, v_2, w_1, w_2)$ between $X_{III} (v_1, v_2)$ and $X_{III} (w_1, w_2)$, where $\phi$ is a formula over $\m Q(t) ^{alg}$. 

We consider the following cases: 
\begin{enumerate} 
\item $\text{td} ( (w_1,w_2) / \m Q (v_1,v_2)) =2$.  
\item $\text{td} ( (w_1,w_2) / \m Q (v_1,v_2)) =1$. 
\item $\text{td} ( (w_1,w_2) / \m Q (v_1,v_2)) =0$. 
\end{enumerate} 
In the first case, $\phi (x,y, v_1, v_2, w_1, w_2)$ is a bijection between $X_{III} (v_1, v_2)$ and $X_{III} (w_1, w_2)$ on a Zariski open subset of $\m C^4$ over $\m Q$ since $(v_1, v_2, w_1, w_2)$ is generic over $\m Q$. This set includes some point such that $v_2 - v_1 -1 \in 2 \m Z$ but $w_2 - w_1 -1 \notin 2 \m Z$ and $w_1 + w_2 +1 \notin 2 \m Z$. But Theorem \ref{Murata3}  shows that there for such values of $(v_1,v_2)$, $X_{III} ( v_1, v_2 ) $ has an algebraic solution and $X_{III} (w_1, w_2)$ has no algebraic solution. This is a contradiction, so $X_{III} ( v_1, v_2 ) $ is orthogonal to $X_{III} ( w_1, w_2)$ in the first case. 

In the second case, we have that either $w_1 + w_2 +1$ or $w_2 - w_1 -1$ is transcendental over $\m Q (v_1,v_2).$ In either case, there is a specialization of the point $(v_1, v_2, w_1, w_2)$ to one in which $w_1 + w_2 +1$ or $w_2 - w_1 -1$ is an even integer while $v_1, v_2 $ are algebraically independent and transcendental. This again contradicts Murata's classification of algebraic solutions, Theorem \ref{Murata3}. 

Set $d_1 = \frac{w_1 + w_2 +1}{2}$, $d_2= \frac{w_2 - w_1 -1}{2}$, $c_1 =\frac{v_1+v_2 +1}{2}$, and $c_2= \frac{v_2-v_1-1}{2}$. Notice that the number of algebraic solutions to $X_{III} (v_1,v_2)$ is controlled by whether or not $c_1, c_2$ are integral. Then in the last case, the tuples $(c_1,c_2)$ and $(d_1, d_2)$ are interalgebraic over $\m Q$, so there are polynomials $p(x_1,x_2,y_1)$ and $q (x_1, x_2, y_2)$ over $\m Q$ such that $p(c_1,c_2,d_1)=0$ and $q (c_1,c_2,d_2)=0$. For a Zariski open subset of $(c_1, c_2) \in \m C^2$, we have that the formula $\phi (x,y, v_1, v_2, w_1, w_2)$ gives a bijection between $X_{III} (v_1, v_2)$ and $X_{III} (w_1, w_2)$.

Now we apply Murata's classification, Theorem \ref{Murata3}, and the fact that a $\m C(t)^{alg}$-definable bijection preserves the number of algebraic solutions. 
So, we have that for all but finitely many $c_1 \in \m  Z$, $p,q$ have the property that for all but finitely many $c_2 \in \m Z$, the following hold:
\begin{itemize} 
	\item Both $p(c_1, c_2, y_1) =0$, $q(c_1, c_2, y_2) =0$ have only integral solutions. 
	\item If $d_1,d_2$ are integral solutions to the respective equations, then as long as $x_2$ appears in the polynomials, $p (c_1, x_2, d_1)=0$ and $q(c_1, x_2, d_2)=0$ have only integral solutions. 
	\end{itemize} 
Now we consider several cases: 
	\begin{enumerate}[label=(\Alph*)]
	\item  The variable $x_2$ appears nontrivially in neither of the equations $p (c_1, x_2, d_1)=0$ and $q(c_1, x_2, d_2)=0$ for generic $c_1, d_1$. 
	\item The variable $x_2$ appears nontrivially in precisely one of $p (c_1, x_2, d_1)=0$ and $q(c_1, x_2, d_2)=0$ for generic $c_1, d_1$. 
	\item  The variable $x_2$ appears nontrivially in both of $p (c_1, x_2, d_1)=0$ and $q(c_1, x_2, d_2)=0$ for generic $c_1, d_1$. 
	\end{enumerate} 
	 Case (A) can not occur since we are assuming that $w_1,w_2$ are algebraic over $\m Q (v_1, v_2)$. In Cases (B) and (C), we assume without loss of generality that $x_2$ appears nontrivially in $p (c_1, x_2, d_1)=0$. Then by the same arguments of Proposition \ref{genone} using the Hilbert Irreducibility theorem, since we may assume without loss of generality that $p$ is monic in $y_1$ and irreducible, we have that 
$p(x_1, x_2 ,y_1)$ must be one of the following: 
	\begin{enumerate} 
		\item $y_1 - x_2 + n$
	\item $y_1 + x_2 +n$
	\item $y_1 - x_2 - x_1 + n $
	\item $y_1 - x_2 + x_1 + n $
	\item $y_1 + x_2 - x_1 +n $
	\item $y_1 + x_2 + x_1 +n$
\end{enumerate} 
for some $n \in \m Z$. 

If we are in case (B), above, then $q(x_1,x_2,y_2)$, by arguments similar to those given above, and the fact that $x_2$ is assumed to not appear in $q$, must be one of: 
\begin{itemize} 
\item $y_2 - x_1 + n_1 $
\item $y_2 + x_1 +n_1 $ 
\end{itemize} 
for some $n_1 \in \m Z$. Now, suppose we are in one of the cases besides (1) or (2). Then the Zariski-open set of $\m C ^2 $ on which $\phi$ gives a bijection between $X_{III} (v_1, v_2) $ and $X_{III} (w_1, w_2)$ contains some points $v_1, v_2$ such that $c_2 \in \m Z$ and $c_1 \notin \m Z$. In this case, we have both $d_1,d_2 \notin \m Z$. This is a contradiction since then $X_{III} (w_1, w_2)$ has no algebraic solutions while $X_{III} (v_1, v_2) $ has two algebraic solutions. In cases (1) and (2), the proposition holds. 

If we are in Case (C), we may argue that $q(x_1,x_2,y_2)$ is in one of the same six forms as given for $p$ with $y_2$ in place of $y_1$ (with an identical argument substituting $y_2$ for $y_1$):

	\begin{enumerate} [label=(\theenumi .a)]
		\item $y_2 - x_2 + n$
		\item $y_2 + x_2 +n$
		\item $y_2 - x_2 - x_1 + n $
		\item $y_2 - x_2 + x_1 + n $
		\item $y_2 + x_2 - x_1 +n $
		\item $y_2 + x_2 + x_1 +n$
	\end{enumerate} 
In all of the cases, in which we have one of (1)-(6) and one of (1.a)-(6.a), as we will explain, there is a Zariski-open set of $\m C ^2 $ on which $\phi$ gives a bijection between $X_{III} (v_1, v_2) $ and $X_{III} (w_1, w_2)$ for which $c_2 \notin \m Z$, $c_1 \in \m Z$ and $d_1,d_2 \notin \m Z$. This gives a contradiction, since $X_{III} (v_1, v_2) $ has two algebraic solutions and $X_{III} (w_1, w_2)$ has no algebraic solutions. And so we can then assume that Case (C) does not occur. 

Suppose that $(1)$ or $(2)$ holds. Then since our tuples are assumed to be interalgebraic, we can assume that $(1.a)$ and $(2.a)$ do not occur. So, assume that one of $(3.a)$ through $(6.a)$ holds. Then pick $c_2 \in \frac{1}{2} + \m Z$ and $c_1 \in \m Z.$ Then $d_1,d_2 \notin \m Z.$ Similarly, if one of $(3)$-$(6)$ holds then again picking $c_2 \in \frac{1}{2} + \m Z$ and $c_1 \in \m Z$ implies that $d_1,d_2 \notin \m Z.$ So, we can see that Case (C) does not occur, and we have established all but the last sentence of the Proposition. 

Finally, if the algebraic relation in question is not induced by a B\"acklund transformation, we would obtain (by composing with the inverse of the algebraic relation associated with the B\"acklund transformation), a nontrivial algebraic relation between $X_{III} (v_1, v_2)$ and itself. But such a relation does not exist for $(v_1, v_2)$ generic over $\m Q$ by \cite{nagloo2014algebraic}. 
\end{proof} 
\begin{ques} Does Proposition \ref{P3orth} hold for non generic coefficients as well? 
\end{ques} 


\section{Painlev\'e four} 
Most of the differential algebraic information in this section comes from \cite{umemura1997solutions}, whose notation we also follow. The fourth Painlev\'e family of equations is given by \[ y ''  = \frac{1}{2y}(y')^2 + \frac{3}{2}y^3 + 4tq^2 + 2 (t^2 - \alpha) y + \frac{\beta }{y},\]  where $ \alpha , \beta $ range over the constants. 
Let $S_{IV} ( v_1, v_2, v_3)$ be the solution set to the following system of differential equations: 

\begin{eqnarray*}
q ' &=& 2pq -q^2 -2tq + 2 ( v_1 -v_2) \\
p' &=& 2pq-p^2 +2tp+ 2(v_1-v_3)
\end{eqnarray*}
where $\bar v = (v_1, v_2, v_3 )$ are constants such that $\bar v \in V := \{ \bar v = \, | \, v_1+v_2+v_3=0 \} $. Then when $ \alpha = 3v_3 +1 , \beta = -2 (v_2 -v_1 )^2$ the elements $q$ such that there is a $p$ so that $(q,p ) $ in the solution set of $ S_{IV} ( v_1, v_2, v_3)$ is precisely the solution set of the fourth Painlev\'e equation. So, we will work with the solution set of $S_{IV} ( v_1, v_2, v_3)$, which we denote by $X_{IV} ( v_1, v_2, v_3)$. 

Define the following affine transformations: 
\begin{eqnarray*} 
s_1 (v_1, v_2, v_3) & =& (v_2, v_1, v_3)\\
s_2 (v_1, v_2, v_3) &=& (v_3,v_2,v_1)	\\
 t_{-}  (v_1, v_2, v_3) &=&  (v_1, v_2, v_3)  + \frac{1}{3} (-1,-1,2) \\ 
 s_0 &=& t_- ^{-1} s_1 s_2 s_1 t_-
  \end{eqnarray*} 

Direct calculations allow one to verify that $s_0 ( v_1, v_2,v_3) = (v_1, v_3+1 , v_2 -1)$. Let $H$ be the subgroup generated by $s_0,s_1,s_2$. It is easy to see that $s_1$ and $s_2$ generate the natural action of the symmetric group $S_3$ on $V$. So, the transformations in $H$ are generated by permutations of the coordinates together with adding one to any coordinate while subtracting one from any other coordinate. 

For any $g \in H$, we have the a rational bijective map $g_*$ so that the following diagram commutes:
\[ \begin{CD}
X_{IV} @>g_*>> X_{IV}\\
@VV \pi V @VV \pi V\\
V @> g  >> V
\end{CD}\]

  Let $\Gamma $ be the subset of $\m C^3$ such that 
  \begin{eqnarray*} 
  Re( v_2 -v_1 ) & \geq & 0 \\
   Re( v_1 -v_3 ) & \geq & 0 \\
    Re( v_3 -v_2 +1 ) & \geq & 0 \\
    Im ( v_2 -v_1 ) & \geq &  0 \, \text{  if } Re (v_2-v_1)=0 \\
     Im ( v_1 -v_3 ) & \geq & 0 \,  \text{  if } Re( v_1 -v_3 ) =0 \\
     Im ( v_3 -v_2 ) & \geq &  0 \, \text{  if } Re( v_3 -v_2 +1 )=0
  \end{eqnarray*}
The set $\Gamma$ is a fundamental region of $V$ for the group $H$. For parameters $\bar v, \bar w$ which are in the same orbit under $H$, the sets $X_{IV} (\bar v) $ and $X_{IV} (\bar w) $ are isomorphic, so, to analyze the fourth Painlev\'e family, it is only necessary to analyze those $\bar v \in \Gamma$. 

Define \[ W = \{ \bar v \in V \, | \, v_1 -v_2 \in \m Z \text{ or } v_3 -v_2 \in \m Z \text{ or } v_1 -v_3  \in \m Z \}.\] Corollaries 3.5 and 3.9 of \cite{umemura1997solutions} imply that $X_{IV} (\bar v)$ is strongly minimal when $\bar v \in \Gamma \setminus W$. 

Define \[ D = \{ \bar v \in V \, | \, v_1 -v_2 \in \m Z \text{ and } v_3 -v_2 \in \m Z \text{ and } v_1 -v_3  \in \m Z \}.\] Then noting that $D$ is the orbit of the origin under $H$, Lemma 3.11 \cite{umemura1997solutions} implies that for $\bar v \in D$, $X_{IV}(\bar v)$ has two irreducible order one differential subvarieties over any differential field $K$ extending $\m C (t)$. So, $X_{IV}(\bar v)$ has Morley rank one and Morley degree three. If $\bar v \in W \setminus D$, then Lemma 3.10 \cite{umemura1997solutions} implies that $X_{IV} (\bar v)$ has one irreducible order one differential subvariety, and so $X_{IV} (\bar v)$ has Morley rank one and Morley degree two. 

\begin{prop} \label{P4orth} Let $\bar v = (v_1, v_2, v_3)$ and $\bar w = ( w_1, w_2, w_3) $ be elements of $V$ such that \[ \text{td} (\bar v / \m Q ) = 2.\] Then $X_{IV} (\bar v) $ is nonorthogonal to $X_{IV} ( \bar w) $ if and only if there is a permutation $\sigma$ of $\{1,2,3\}$ such that for $i=1,2,3$, $v_i - w_ {\sigma (i) } \in \m Z$.  In this case, the only algebraic relations between solutions of the equations are induced by B\"acklund transformations. 
\end{prop} 
\begin{proof} We will only sketch the result since it is very similar to the proof of Proposition \ref{P3orth}. By the strong minimality and triviality of $X_{IV} (\bar v) $ and $X_{IV} ( \bar w) $ \cite{nagloo2017algebraic} together with Lemma \ref{GST}, we have a $\m Q(t)$ formula $\phi ( x,y , \bar v , \bar w)$ which defines a bijection between $X_{IV} (\bar v) $ and $X_{IV} ( \bar w) $. Now the mechanics of the proof proceed as in \ref{P3orth}, because no such bijection can exist between $X_{IV} (\bar v) $ and $X_{IV} ( \bar w) $ if $\bar v, \bar w$ come from different sets among $V \setminus W , W \setminus D , D$ (because in this case $X_{IV} (\bar v) $ and $X_{IV} ( \bar w) $ have different Morley degree).\footnote{It is also possible to argue here by using the classification of algebraic solutions - see \cite{Murata}.} Again, the final assertion follows as in Proposition \ref{genone} - if the algebraic relation in question is not one associated with the suitable B\"acklund transformation, then composing the two relations results in a nontrivial algebraic relation between $X_{IV} (\bar v)$ and itself, which can not happen by \cite{nagloo2014algebraic}. 
\end{proof} 
\begin{ques} 
Does Proposition \ref{P4orth} hold without the assumption that one of the equations has generic coefficients? 
\end{ques} 

Next, we establish one of the two remaining open cases of Theorem \ref{Naggy}.

	\begin{prop} \label{Naggy1} Let $\alpha$ be transcendental, and let $\bar v \in V$ be generic. Then $X_{II} ( \alpha)$ is orthogonal to $X_{IV} ( \bar v)$. 
		\end{prop} 
		\begin{proof} By \cite{nagloo2017algebraic}, both $X_{II} ( \alpha)$ and $X_{IV} ( \bar v)$ are strongly minimal and strictly disintegrated (thus geometrically trivial), so by Lemma \ref{GST}, if the two sets are nonorthogonal, they are non weakly orthogonal, and there is a formula $\phi (x,y, \alpha , \bar v ) $ with parameters from $\m Q (t) ^alg$ which gives a definable bijection between $X_{II} ( \alpha)$ and $X_{IV} ( \bar v)$. By quantifier elimination and stable embeddedness of the constants, $\phi (x,y, z_0 , \bar z_1 ) $ gives a bijection between $X_{II} ( z_0)$ and $X_{IV} ( \bar z_1)$ on some Zariski open subset of the locus of $( \alpha , \bar v )$ over $\m Q ^{alg}$. This locus must, by the assumption that $\bar v \in V$ is generic, project dominantly onto $V$. 
		
Thus, because the set $D \subset V$ is Zariski dense, $\phi (x,y, z_0 , \bar z_1 ) $ gives a bijection between $X_{II} ( z_0)$ and $X_{IV} ( \bar z_1)$ for some $ \bar z_1 \in D$. This is impossible, since by the results of \cite{umemura1997solutions} interpreted above, the Morley degree of $X_{IV} ( \bar z_1)$ is three and the Morley degree of $X_{II} ( z_0)$ is at most two as established in Section \ref{oneone}. 
		\end{proof} 

As Nagloo notices \cite{nagloo2015transformations}, by Murata's results, \cite{Murata}, the structure of the algebraic solutions of $P_{II}$ and $P_{IV}$ are very similar. In fact, the stucture of the relations is such that the arguments of \cite{nagloo2015transformations} do not seem to be adaptable to this case. In differential algebraic terms, our argument is relying on the differences in the structure of the classical solutions of $P_{II}$ and $P_{IV}$, which manifest in the form of differences in Morley degree of the fibers. See \cite{nagloo2017algebraic} for an explanation of classical solutions in model theoretic terms. 
	
\section{Painlev\'e five} \label{P5}
Our notation in this section comes from \cite{watanabe1995solutions}, whose results we will utilize. The fifth Painlev\'e family is equivalent to the following system of equations \begin{eqnarray*} t Q' & = & 2 Q (Q-1)^2 P + (3v_1+ v_2 ) Q^2 -(t+4v_1)Q + v_1 -v_2 \\ tP' &=& (-3Q^2 +4Q -1) P^2 -2(3v_1 +v_2 ) Q P + (t+4v_1) P - (v_3-v_1)(v_4-v_1)
\end{eqnarray*} 
where $\bar v = (v_1, v_2, v_3, v_4) \in \m C^4$ lies in the hyperplane $V$ defined by $\sum v_i = 0$. 
Setting $q = \frac{Q}{Q-1 } $ and $p = -(Q-1)^2 P + (v_3-v_1 )(Q-1)$ then $p$ and $q$ satisfy 
\begin{eqnarray*} t q' & = & 2 q^2 p - 2 qp + t q^2 -tq + ( v_1-v_2-v_3+v_4) q + v_2 -v_1 \\ 
t p' &=& -2 qp^2 + p^2 -2tpq +tp -(v_1-v_2-v_3+v_4) p + (v_3-v_1)t
\end{eqnarray*} 
and solutions to this system are birational with the solutions to our earlier system. The properties we study are not sensitive up to birationality, so we will work with this second system, whose solution set we denote by $X_V ( \bar v)$. Let $X_V$ denote \[\{(p,q, \bar v ) \, | \, (p,q) \in X_V (\bar v ), \, \bar v \in V \},\]
and let $\pi_2$ denote the natural map from $X_V \rightarrow V$ given by projection. 
Let \[W = \left\{ \bar v \in V \, | \, \text{for some } 1 \leq i \leq j \leq 4, \, v_i - v_j \in \m Z \right\}.\] Corollary 2.6 of  \cite{watanabe1995solutions} implies that for $\bar v \notin W$, $S( \bar v )$ is strongly minimal. In particular, for generic parameters, $X_V( \bar v)$ is strongly minimal. Lemmas 3.1-3.4 of \cite{watanabe1995solutions} imply that for $\bar v \in W$ the Morley rank of $X_V (\bar v)$ has Morley rank one and Morley degree between two and four; the specific loci with a given degree can be deduced from the cited lemmas and noting that a group of affine transformations specified in \cite{watanabe1995solutions} acts on the family of equations. First we describe the group of transformation, and then we will give the specify the set of parameters with a fixed degree. 

Let 
\begin{eqnarray*} 
	s_1 ( v_1, v_2, v_3 ,v_4 ) & := & (v_2, v_1, v_3, v_4) \\
	s_2 ( v_1, v_2, v_3 ,v_4 ) & := & (v_3, v_2, v_1, v_4) \\
	s_3 ( v_1, v_2, v_3 ,v_4 ) & := & (v_1, v_2, v_4, v_3) \\
	t_{-} ( v_1, v_2, v_3 ,v_4 ) & := & (v_1, v_2, v_3, v_4) + \frac{1}{4} (-1,-1,-1,3) 
\end{eqnarray*}
Let $G$ be the group of transformations of $V$ generated by $s_1, s_2, s_3, t_-$. 
Then, for instance, \[t_-^{-1} s_3 s_1s_2 s_1 s_3 t _- (v_1, v_2,v_3,v_4) = (v_1, v_4+1 , v_3 , v_2 -1),\] and $\langle s_1, s_2, s_3 \rangle = S_4$. For each $\tau \in G$, there is a birational bijective map (defined over $\m Q$) $\tau_*: X_V \rightarrow X_V$ such that the following diagram commutes: 
\[ \begin{CD}
X_V @>\tau_*>> X_V\\
@VV \pi_2 V @VV \pi_2 V\\
V @> \tau  >> V
\end{CD}.\] 
The transformations are constructed (in slightly different notation) in \cite{okamoto1986studies}, and are given in \cite{watanabe1995solutions} similar notation to ours. Thus, the sets $X_V (\bar v_1)$ and $ X_ V ( \bar v_2) ,$ where $\bar v_1, \bar v_2$ are such that there is $\tau \in G$ such that $\tau (\bar v_1) = \bar v_2$, are nonorthogonal (there is a $\m Q$-definable bijection between the sets). We next aim to characterize the orbits of the action of $G$ on $V$ in order to state our results on nonorthogonality. 

The group $G$ contains each of the following transformations: 
\begin{eqnarray*} 
	\alpha_1  ( v_1, v_2, v_3 ,v_4 ) & := & ( v_1, v_2, v_3 ,v_4 ) + \frac{1}{4} (1,1,1,-3) \\
	\alpha_2  ( v_1, v_2, v_3 ,v_4 ) & := & ( v_1, v_2, v_3 ,v_4 ) + \frac{1}{4} (0,0,-4,4) \\
	\alpha_3  ( v_1, v_2, v_3 ,v_4 ) & := & ( v_1, v_2, v_3 ,v_4 ) + \frac{1}{4} (0,-4,0,4) \\
	\alpha_4  ( v_1, v_2, v_3 ,v_4 ) & := & ( v_1, v_2, v_3 ,v_4 ) + \frac{1}{4} (-4,0,0,4). \\
\end{eqnarray*}
Let $H$ be the $\m Z$-module generated by $\alpha_1, \ldots , \alpha_4$. It can be verified that any element of $G$ written as $\sigma \tau $ where $\tau \in H$ and $\sigma \in S_4$ can be rewritten as $\tau_1 \sigma $ for appropriately chosen $\tau _1 \in H$. So, two elements $\bar v , \bar w \in V$ are related by an element of $G$ if there is a permutation $\sigma $ and some elements $\tau \in H$ such that 

\[\tau \left(v_{\sigma (1)} , v_{\sigma (2)}, v_{\sigma (3)}, v_ {\sigma (4)} \right) = \left(w_1,w_2,w_3, w_4 \right).\] 

This last condition is equivalent to there existing $ a \in \m Z$ such that:
\[\frac{a}{4} \left( 1,1,1.-3 \right) + \left(v_{\sigma (1)} , v_{\sigma (2)}, v_{\sigma (3)}, v_ {\sigma (4)} \right) - \left(w_1,w_2,w_3, w_4 \right)  \in \m Z^4.\] 

Of course, since the the elements $\bar v$ of $V$ satisfy that $\sum v_i =0$, the last condition is equivalent to:
\[\frac{a}{4} \left( 1,1,1 \right) + \left(v_{\sigma (1)} , v_{\sigma (2)}, v_{\sigma (3)}\right) - \left(w_1,w_2,w_3\right)  \in \m Z^3.\] 

Next, we describe the results of \cite{watanabe1995solutions} in model theoretic terms. Let $W_1$ denote the set of $\bar v \in V$ such that for some $\sigma \in S_4,$ $v_{\sigma (1)} - v_{\sigma (2)}  \in \m Z$ and $v_{\sigma (3)} - v_{\sigma (4)}  \in \m Z.$ Let $W_2$ denote the set of $\bar v \in V$ such that for some $\sigma \in S_4,$ $v_{\sigma (1)} - v_{\sigma (2)}  \in \m Z$ and $v_{\sigma (2)} - v_{\sigma (3)}  \in \m Z.$ Noting the action of the group $G$, Lemma 3.1 of \cite{watanabe1995solutions} implies that for $\bar v $ such that $\bar v \in W$, but $\bar v \notin W_1 \cup W_2$ , the set $X_V(\bar v)$ is of Morley rank one and Morley degree two. Let $D$ denote the subset of $\bar v \in V$ such that each of the entries of $ \bar v$ is in the same $\m Z$-coset. Again noting the action of $G$, Lemma 3.2 of \cite{watanabe1995solutions} implies that for $\bar v  \in W_2$, but $\bar v \notin D$, $X_V(\bar v)$ has Morley rank one and Morley degree three. Lemma 3.3 of \cite{watanabe1995solutions} implies that for $\bar v$ such that $\bar v \in W_1$ and $\bar v \notin D$, $X_V( \bar v)$ has Morley rank one and Morley degree three. Lemma 3.4 of \cite{watanabe1995solutions} implies that for $\bar v \in D$, $X_V(\bar v)$ has Morley rank one and Morley degree four. By \cite{watanabe1995solutions}, Theorem 0.4, for $\bar v \in W_2 \setminus D$, $X_V(\bar v)$ has one algebraic solution. If $\bar v \in D,$ $X_V ( \bar v ) $ has two algebraic solutions. 

Armed with the the above information, we now establish Theorem \ref{Naggy} by covering the last remaining case: 
	
	\begin{prop} \label{Naggy2} Let $\alpha,\beta$ be transcendental and independent, and let $\bar v \in V$ be generic. Then $X_{III} ( \alpha, \beta)$ is orthogonal to $X_{V} ( \bar v)$. 

		\end{prop} 
	\begin{proof} By \cite{nagloo2017algebraic}, both $X_{III} ( \alpha, \beta )$ and $X_{V} ( \bar v)$ are strongly minimal and strictly disintegrated (thus geometrically trivial). Then, by Lemma \ref{GST}, if the two sets are nonorthogonal, there must be a formula $\phi (x,y, \alpha, \beta , \bar v ) $ with parameters from $\m Q (t) ^alg$ which gives a definable bijection between $X_{III} ( \alpha, \beta)$ and $X_{V} ( \bar v)$. By quantifier elimination and stable embeddedness of the constants, $\phi (x,y, w_1,w_2 , \bar z ) $ gives a bijection between $X_{III} ( w_1, w_2)$ and $X_{V} ( \bar z)$ on some Zariski open subset of the locus of $( \alpha, \beta , \bar v )$ over $\m Q ^{alg}$. This locus must project dominantly onto $V$. 
	
	As the set $D$ is dense in $V,$ we must have such a bijection for some $\bar z \in D.$ But then by Lemma 3.4 of \cite{watanabe1995solutions} as interpreted above, $X_{V} ( \bar z)$ has Morley degree four, while by \cite{umemura1998solutions} as interpreted in Section \ref{twotwo}, any equation in the third Painlev\'e family has Morley degree at most three. This is a contradiction as definable bijections preserve Morley degree. 
	\end{proof} 
	
	The B\"acklund transformations detailed above each give rise to instances of nonorthogonality between fibers of the the fifth Painlev\'e family, and next we turn our attention towards towards classifying algebraic relations within the family. 

\begin{ques} \label{ques5} If $\bar v, \bar w \in V$ are such that there is no $\sigma \in S_4$ and $a \in \m Z$ such that \[\frac{a}{4} \left( 1,1,1 \right) + \left(v_{\sigma (1)} , v_{\sigma (2)}, v_{\sigma (3)}\right) - \left(w_1,w_2,w_3\right)  \in \m Z^3,\] then are $X_V (\bar v )$ and $X_V ( \bar w)$ orthogonal? 
\end{ques}
As in previous sections, our techniques can be used to answer the question for generic values of the parameters. Similar to the other families of equations, additional cases which we do not state specifically are accessible via our methods, which rely on specialization arguments which exploit the structure of the exceptional fibers of Morley degree larger than one. 

\begin{prop} \label{P5ortho} Let $\bar v \in V$ be generic over $\m Q$. Let $\bar w \in V$. Then $X_V ( \bar v)$ is orthogonal to $X_V ( \bar w)$ unless there is $\sigma \in S_4$ and $a \in \m Z$ such that \[\frac{a}{4} \left( 1,1,1.-3 \right) + \left(v_{\sigma (1)} , v_{\sigma (2)}, v_{\sigma (3)}, v_ {\sigma (4)} \right) - \left(w_1,w_2,w_3, w_4 \right)  \in \m Z^4.\]  In this case, the only algebraic relations between solutions of the equations are induced by B\"acklund transformations. 
	\end{prop} 
	\begin{proof} 
	By \cite{nagloo2014algebraic}, since $X_V (\bar v )$ is strictly disintegrated for $\bar v $ generic over $\m Q$ if $X_V ( \bar v)$ is nonorthogonal to $X_V ( \bar w)$, then this relation must be witnessed by a formula $\phi (x,y, \bar v , \bar w , t)$ which can be taken to be over $\m Q$ by Lemma \ref{GST}. It follows by the  disintegratedness of $X_V (\bar v )$ that $\phi (x,y, \bar v , \bar w , t)$ must define a generically finite-to-one map from $X_V ( \bar w)$ to $X_V (\bar v )$. 
	
	 We now consider two cases: 
\begin{itemize} 
\item Case I: $\bar v$ and $\bar w$ are not interalgebraic over $\m Q$.
\item Case II: $\bar v$ and $\bar w$ are interalgebraic over $\m Q$. 
\end{itemize} 
Suppose we are in Case I. Now, there are two subcases: 
\begin{itemize} 
\item Case Ia: There are two elements of $\bar w$ whose difference is an integer. In this case, the Morley degree of $X_V (\bar w)$ is greater than one, and it must be the case the the map defines a $n$-to-one map from $X_V (\bar w) \setminus S$ to $X_V (\bar v)$, where $S$ is a finite union of order one subvarieties - note that the equations defining $S$ depend only on the given integral differences (and the coordinates which differ by integers). This statement must be true on an open Zariski-dense subset $U$ of $(\bar w_0, \bar v_0)$ in the locus of $(\bar w, \bar v)$ over $\m Q$. Thus, noting that $D$ is dense in $V$, it must hold for some $(\bar w_0, \bar v_0) \in ( \m V \times D) \cap U$. Note that the fact that $(\bar w_0, \bar v_0)$ is in the locus of $(\bar w, \bar v)$ over $\m Q$ implies that $S$ is an order one subvariety of $X_V (\bar w)$. But we have a $n$-to-one $X_V (\bar w_0) \setminus S$ to $X_V (\bar v_0)$, but $X_V (\bar w_0) \setminus S$ has Morley degree at most $3$ since $S$ is an order one subvariety, while $X_V (\bar v_0)$ has Morley degree $4$, a contradiction. 

\item Case Ib: The difference between any two elements of $\bar w$ is not an integer. In this case, the result follows similarly to Case Ia after noting that in any Zariski open set containing $(\bar w , \bar v)$ we must have a point $(\bar w, \bar v_0)$ such that there are two coordinates of $\bar v_0$ which differ by an integer. 
\end{itemize} 
Now we consider case II. Here $\bar w$ must be generic in $V$ over $\m Q.$ It follows that $\phi (x,y, \bar v_0 , \bar w_0 , t)$ defines a bijection between $X_V ( \bar v_0)$ and $X_V ( \bar w_0)$ for a Zariski-dense open set of $(\bar v_0, \bar w_0)$ of $X,$ the locus of $(\bar v, \bar w)$ over $\m Q.$ We denote this open set by $U$. 

Set $c_i: = v_i -v_{i+1}$ for $i=1,2,3$ and $d_i: = w_i -w_{i+1}$ for $i=1,2,3.$ It follows from our assumptions (being in Case II) that $\bar c $ is interalgebraic with $\bar d$ over $\m Q$. So, for each $i=1,2,3,$ we have some polynomial over $\m Q$ with $p_i (\bar c ,  d_i) = 0$. Since Morley degree is preserved by definable bijections and the Morley degree of equations in this family is maximal precisely when all parameters are in the same $\m Z$-coset, for some Zariski open $U_1 \subset \m A^3,$ if $\bar c \in U \cap \m Z^3$ we must have that $p_i ( \bar c , d_i)$ for $i=1,2,3$ has only integer solutions $d_i$ (note here that $U_1$ is the projection of $U$ from the previous paragraph); similarly for $\bar c$ integral in a suitable open set $U_2,$ there are only integral solutions $\bar d$ to the system.  

As in the proofs of Propositions \ref{genone} and \ref{P3orth}, by the Hilbert Irreducibility theorem we must have that $$p_i ( \bar c , d_i)  =d_i + \sum_{j=1}^3 \epsilon _{i,j} c_j +n_i$$ for some $\epsilon _{i,j} \in \{-1,0,1 \}$ and some $n_i \in \m Z.$  

Now, fix $c_1 \in \m Z$ such that the fiber of $U_2$ above $c_1$ is dense in $\m C^2.$ Now, the argument from Cases (A)-(C) of the proof of Proposition \ref{P3orth} applies to the pairs of polynomials $p_i$, $p_j$, and it follows that in and given $p_i$ only $c_2$ or $c_3$ can appear nontrivially, but not both. 

One can repeat the argument of the above paragraph with $c_2$ and $c_3$ in place of $c_1$ to see that in each $p_i$ only one of $c_1,c_2,c_3$ can appear. So, perhaps after permuting the $c_j$, we can assume that $p_i$ is given by $d_i + \epsilon _i c_i +n_i$ where $\epsilon _i$ is $\pm 1.$

Simplifying and applying the relation $\sum _{i=1}^4 v_i = \sum _{i=1}^4 w_i =0$ we obtain the following system: \begin{align*}
\left( \begin{matrix}
1 & -1 & 0 & 0\\
0 & 1 & -1 & 0\\ 
1 & 1 & 2 & 0 \\
0 & 0 & 0 & 1
\end{matrix}\right) \left( \begin{matrix}
v_1 \\
v_2 \\
v_3 \\
1
\end{matrix}\right) = \left( \begin{matrix}
\epsilon_1 & -\epsilon_1 & 0 & n_1\\
0 & \epsilon_2 & -\epsilon_2 & n_2\\ 
\epsilon_3 & \epsilon_3 & 2\epsilon_3 & n_3 \\
0 & 0 & 0 & 1
\end{matrix}\right) \left( \begin{matrix}
w_1 \\
w_2 \\
w_3 \\
1
\end{matrix}\right)
\end{align*} 

and so after inverting the matrix on the left and multiplying, we obtain: 
 \begin{align*}
    \left( \begin{matrix}
v_1 \\
v_2 \\
v_3 \\
1
\end{matrix}\right) =
 \frac{1}{4} & \left( \begin{matrix}
3 \epsilon_1 +  \epsilon_3& -3  \epsilon_1 +2  \epsilon_2 +  \epsilon_3 & -2 ( \epsilon_2 -  \epsilon_3) & 3n_1 +2 n_2 +n_3 \\
- \epsilon_1 +  \epsilon_3 &  \epsilon_1 + 2  \epsilon_2 + \epsilon_3 & -2 ( \epsilon_2 - \epsilon_3 ) & -n_1 +2 n_2 +n_3\\ 
- \epsilon_1 + \epsilon_3 &  \epsilon_1 -2  \epsilon_2 + \epsilon_3 & 2( \epsilon_2 + \epsilon_3) & -n_1 -2n_2+n_3\\
0 & 0 & 0 & 4
\end{matrix}\right) \left( \begin{matrix}
w_1 \\
w_2 \\
w_3 \\
1
\end{matrix}\right)
\end{align*}
It is clear that we must have $3n_1+2n_2+n_3 = -n_1 +2n_2+n_3 = -n_1 -2n_2 +n_3$ modulo $4,$ since otherwise there is a Zariski dense set of integers $w_1, w_2,w_3$ such that $v_1, v_2, v_3$ do not each differ by an integer. By arguments above, this is impossible. 

If we have that $\epsilon_1= 1, \epsilon_2=-1, \epsilon_3 = -1$, then the matrix becomes: 
\begin{align*}
\left( \begin{matrix}
\frac{1}{2} & \frac{-3}{2} & 0 & \frac{3n_1 +2 n_2 +n_3}{4} \\
\frac{-1}{2} & \frac{-1}{2} & 0 & \frac{-n_1 +2 n_2 +n_3}{4}\\
\frac{-1}{2} & \frac{1}{2} & \frac{-1}{2} & \frac{-n_1 -2n_2+n_3}{4} \\
0 & 0 & 0 & 1
\end{matrix}\right) 
\end{align*} 

Recall that we can assume that $\frac{3n_1 +2 n_2 +n_3}{4}=\frac{-n_1 +2 n_2 +n_3}{4}=\frac{-n_1 -2n_2+n_3}{4}$ as an element of $\m R /\m Z.$ Call this element $a$. In this case, picking $w_1$ even and $w_2,w_3$ odd integers gives, after translating by $-a$, $v_1, v_2 \in \frac{1}{2}+\m Z$ and $v_3 \in \m Z$. Since the set of such $\bar w$ is dense, this gives a contradiction as above since Morley rank is an invariant of definable bijections. Similarly, one obtains a nearly identical contradiction when $\epsilon_1= -1, \epsilon_2=1, \epsilon_3 = 1$. So, this case is also impossible. 

If we have that $\epsilon_1= 1, \epsilon_2=-1, \epsilon_3 = 1$ the matrix becomes: 
\begin{align*}
\left( \begin{matrix}
1 & -1 & -1 & \frac{3n_1 +2 n_2 +n_3}{4} \\
0 & \frac{-1}{2} & -1 & \frac{-n_1 +2 n_2 +n_3}{4}\\
0 & \frac{1}{2} & 0 & \frac{-n_1 -2n_2+n_3}{4} \\
0 & 0 & 0 & 1
\end{matrix}\right) 
\end{align*} 

In this case, if $w_2$ is odd, then after translation by $-a$, we obtain $v_1 \in \m Z$ but $v_2, v_3 \in \frac{1}{2}+ \m Z$. A contradiction is obtained as in the previous cases. Similarly, the case $\epsilon_1= -1, \epsilon_2=1, \epsilon_3 = -1$ is also impossible. 

If we have that  $\epsilon_1= 1, \epsilon_2=1, \epsilon_3 = -1$, then our matrix becomes: 
\begin{align*}
\left( \begin{matrix}
\frac{1}{2} & -\frac{1}{2} & -1 & \frac{3n_1 +2 n_2 +n_3}{4} \\
0 & \frac{1}{2} & -1 & \frac{-n_1 +2 n_2 +n_3}{4}\\
0 & -\frac{1}{2} & 0 & \frac{-n_1 -2n_2+n_3}{4} \\
0 & 0 & 0 & 1
\end{matrix}\right) 
\end{align*} 
In this case, picking $w_1$ odd and $w_2$ even and $w_3 $ any integer yields (after translation by $-a$) $v_1 \in \frac{1}{2}+\m Z$ and $v_2,v_3 \in \m Z$, a contradiction as in the previous cases. Finally, the case $\epsilon_1= -1, \epsilon_2=-1, \epsilon_3 = 1$ is almost identical. 

Now we have shown that there are only two possible cases $\epsilon_1=\epsilon_2=\epsilon_3=\pm 1$. In the case that $\epsilon_i = -1,$ then we have a system of the form: 

\begin{align} \label{lastcase5}
    \left( \begin{matrix}
v_1 \\
v_2 \\
v_3 \\
1
\end{matrix}\right) =
\left( \begin{matrix}
-1 & 0 & 0 & \frac{m_1}{4} \\
0 & -1 & 0 & \frac{m_2}{4}\\
0 & 0 & -1 & \frac{m_3}{4} \\
0 & 0 & 0 & 1
\end{matrix}\right)  \left( \begin{matrix}
w_1 \\
w_2 \\
w_3 \\
1
\end{matrix}\right)
\end{align} where $m_1,m_2,m_3 \in \m Z$ and are congruent modulo $4$. We will next show that this is impossible, and so we must have $\epsilon_1=\epsilon_2=\epsilon_3=1$. Suppose not. Then because the generic fiber of $X_V$ is strongly minimal, disintegrated and has no algebraic solutions, there must be a dense set of $\bar w$ such that there is a bijection $\phi$ defined over $\m Q (\bar w , t)$ from $X_V (\bar w )$ and $X_V (\bar v)$ for $\bar v$ as in equation (\ref{lastcase5}). The formula defining $\phi$ is given by a differential rational map of bounded degree, and as such, for an open dense set $U$ of $\bar w$ on which $\phi$ is a bijection, for all sufficiently large $k_1,k_2,k_3 \in \m Z$, we must have $\phi (y) \neq \psi_{-2k_1+\frac{m_1}{4},-2k_2+\frac{m_2}{4},-2k_3+\frac{m_3}{4}}(y)$ for all $y \in X_{V} (\bar w)$ and there are not $y_1, y_2 \in X_V(\bar w )$ such that $\phi (y_1) = \psi_{-2k_1+\frac{m_1}{4},-2k_2+\frac{m_2}{4},-2k_3+\frac{m_3}{4}}(y_2)$ and $\phi (y_2) = \psi_{-2k_1+\frac{m_1}{4},-2k_2+\frac{m_2}{4},-2k_3+\frac{m_3}{4}}(y_1)$ where $\psi_{-2k_1+\frac{m_1}{4},-2k_2+\frac{m_2}{4},-2k_3+\frac{m_3}{4}}$ is the B\"acklund transformation mapping $X_{V} (w_1,w_2,w_3,w_4)$ to $X_{V} (v_1,v_2,v_3,v_4)$ where 
\begin{align*} 
    \left( \begin{matrix}
v_1 \\
v_2 \\
v_3
\end{matrix}\right) =
\left( \begin{matrix}
1 & 0 & 0 & -2k_1+\frac{m_1}{4} \\
0 & 1 & 0 & -2k_3+\frac{m_2}{4}\\
0 & 0 & 1 & -2k_3+\frac{m_3}{4}
\end{matrix}\right)  \left( \begin{matrix}
w_1 \\
w_2 \\
w_3 \\
1
\end{matrix}\right)
\end{align*}
and $\sum v_i = 0.$ To see this last property, note that for generic $\bar w$, solutions of $X_V(\bar w)$ and their first derivative are algebraically independent from each other over $\m C(t)$ by results of \cite{nagloo2017algebraic}, and $y_1,y_2$ with the above property contradict this. The same argument applies with $((z_0)_*)^i \circ \psi_{-2k_1+\frac{m_1}{4},-2k_2+\frac{m_2}{4},-2k_3+\frac{m_3}{4}}$ for $i=1,2,3$ in place of $\psi_{-2k_1+\frac{m_1}{4},-2k_2+\frac{m_2}{4},-2k_3+\frac{m_3}{4}}$ where $(z_0)_*$ is the map defined by Okamato in \cite{okamoto1987studies5} and introduced in section 1 of \cite{watanabe1995solutions}. 

Let $w_1,w_2,w_3 \in \m Z$ and $w_4 = -w_1-w_2-w_3$ be in the open set $U$ and also in the corresponding open sets for $((z_0)_*)^i \circ \psi_{-2k_1+\frac{m_1}{4},-2k_2+\frac{m_2}{4},-2k_3+\frac{m_3}{4}}$ for $i=1,2,3$. By the results of Watanabe referenced above (specifically, here the birational transformations developed before question 5.1 and Lemma 3.4 and Theorem 1.3 of \cite{watanabe1995solutions}) $X_V (\bar w )$ is of Morley rank one and Morley degree $4$ with $3$ subvarieties of order one given as zero sets of Riccati equations. The bijection $\phi$ is a map from $X_V (\bar w )$ to $X_V (\bar v )$ where $\bar v$ is given as in equation (\ref{lastcase5}). The map $\chi:=  \psi_{-2w_1+\frac{m_1}{4},-2w_2+\frac{m_2}{4},-2w_3+\frac{m_3}{4}} \circ \phi^{-1}$ is a bijection from $X_V (\bar w)$ to itself, and by the properties of $\phi$ stipulated above, if we pick $y_1, y_2 \in X_V (\bar w)$, then $y_1, y_2, \chi (y_1) , \chi (y_2)$ are four distinct solutions of $X_V(y_1,y_2)$. Further, choosing $y_1, y_2$ in the same Riccati subvariety of of $X_V( \bar w)$, we must have $\chi (y_1), \chi (y_2) $ in the same Riccati subvariety. Now, by Lemma 1.4 of \cite{watanabe1995solutions}, the map $(z_0)_*$ permutes the four Riccati subvarities in question ($(z_0)_*$ gives a $4$-cycle on the solution sets). So, for some $i,$ we now have four distinct solutions $y_1,y_2,(z_0)_* ^i \circ \chi (y_1), (z_0)_* ^i  \circ \chi (y_2).$ Note that $(z_0)_* ^i \circ \chi (y_j) \in \m C (t) ^{alg}  ( y_j )$ for $j=1,2$. But, this contradicts Proposition 2.5 of \cite{nagloo2019algebraic}; any three distinct solutions of the Riccati equation are algebraically independent. Thus $\epsilon_1=\epsilon_2=\epsilon_3=-1$ is impossible, and so we must have 

\begin{align*} \label{lastcase5}
    \left( \begin{matrix}
v_1 \\
v_2 \\
v_3 
\end{matrix}\right) =
\left( \begin{matrix}
1 & 0 & 0 & \frac{m_1}{4} \\
0 & 1 & 0 & \frac{m_2}{4}\\
0 & 0 & 1 & \frac{m_3}{4}
\end{matrix}\right)  \left( \begin{matrix}
w_1 \\
w_2 \\
w_3 \\
1
\end{matrix}\right)
\end{align*}

with $m_i$ all congruent modulo $4.$ Thus $\bar v$ and $\bar w$ have the desired form. The fact that the map must be the known B\"acklund transformation follows as in the conclusion of the proof of Proposition \ref{genone}.
\end{proof}

\section{Painlev\'e six} \label{P6sec}
\subsection{Painlev\'e six and the Manin map}

In this subsection we will deal with the sixth Painlev\'e equation expressed in the terms used in  \cite{hitchin1992poncelet} and \cite{Manin6}, where the equation is given in the following form: 

\begin{equation} \label{MP6} \tag{$P_{VI}$} \begin{aligned} \frac{d^2 x }{dt^2} &=& \frac{1}{2} \left( \frac{1}{x}+ \frac{1}{x-1} + \frac{1}{x-t} \right) \left( \frac{dx}{dt} \right) ^2 - \left( \frac{1}{t}+ \frac{1}{t-1} + \frac{1}{x-t} \right) \left( \frac{dx}{dt} \right)\\ & & + \frac{x(x-1)(x-t)}{t^2(t-1)^2} \left( \alpha + \beta \frac{t}{x^2} + \gamma \frac{t-1}{(x-1)^2} + \delta \frac{t(t-1)}{(x-t)^2} \right) \end{aligned}  \end{equation}
where $\alpha, \beta, \gamma, \delta \in \m C$ are complex parameters. When the parameters are generic, that is, $td _ {\m Q } \left(\alpha, \beta, \gamma, \delta \right)=4,$ as we proved in Theorem \ref{strongpain}, the equation is strongly minimal. The proof of \cite{watanabe1998birational} establishes additional structural results beyond strong minimality of equation \ref{MP6} with generic parameters, but is quite technical - this work will be discussed in the next subsection. 

The Manin map for the family of Elliptic curves $E_t$ given by the projective closure of $y^2= x(x-1)(x-t)$ is a homomorphism from the elliptic curve to the additive group and has been the subject of intense model theoretic investigation. The Manin map was defined originally in \cite{ManinDiff} in additional generality. The specific formula for the map we use comes from \cite{PongThesis} which slightly corrected a formula of \cite{ManinDiff}. By \cite[Theorem B.10]{PongThesis} we have the following formula for the Manin map: 

$$\mu (x,y) = - \frac{y}{(x-t)^2}+ \left(2t(t-1)\frac{x'}{y}\right) ' + \frac{t(t-1) x'}{x-t)y}.$$

After simplifying the above expression, one can check that the kernel of the map $\mu(x,y))$ is defined by 

\begin{equation} \label{MK1} \tag{MK} \begin{aligned} \frac{d^2 x }{dt^2} &=& \frac{1}{2} \left( \frac{1}{x}+ \frac{1}{x-1} + \frac{1}{x-t} \right) \left( \frac{dx}{dt} \right) ^2 \\ & & - \left( \frac{1}{t}+ \frac{1}{t-1} + \frac{1}{x-t} \right) \left( \frac{dx}{dt} \right) \\& & + \frac{x(x-1)}{2t(t-1)(x-t)} \end{aligned}  \end{equation}

So, the kernel of $\mu(x,y))$ is defined by \ref{MP6} with parameters $(0,0,0,\frac{1}{2})$. By a result of Hrushovski \cite{HrushovskiMordell-Lang}, the Kernel of $\mu(x,y))$ is strongly minimal - see \cite[section 5]{marker2000manin} for a short exposition of this result. 

\subsection{Algebraic relations between solutions}

We will next describe the group of transformations of the family
$P_{VI}$. In this section, we work with the Hamiltonian form of $P_{VI}$ given by:

\[
\begin{array}{rll}
t(t-1)q'&=&2q(q-1)(q-t)p+(v_1+v_2-2v_3)q^2\\
& &+(2v_3t-v_1-v_2+v_3+v_4)q-(v_3+v_4)t\\
t(t-1)p'&=&-3q^2p^2+2(1+t)qp^2-tp^2-2(v_1+v_2-2v_3)qp\\
& &-(2v_3t-v_1-v_2+v_3+v_4)p-(v_3-v_2)(v_3-v_1)\\
\end{array}.\]

where $\alpha=\frac{1}{2}(v_1-v_2)^2$, $\beta =-\frac{1}{2}(v_3+v_4)^2$, $\gamma =\frac{1}{2}(v_3-v_4)^2$ and 
$\delta =-\frac{1}{2}(1-(1-v_1-v_2)^2)$ (cf. \cite{watanabe1998birational}). For $\bar v \in \m C^4$, we will denote by $X_{VI} (\bar v)$ the solution set of the Hamiltonian system.

\begin{rem} The form of $P_{VI}$ studied in the previous section is equivalent to its Hamiltonian form we consider here when $q\neq 0,1,t$. While $P_{VI}(0,0,0,\frac{1}{2})$ is strongly minimal as discussed above, the corresponding Hamiltonian system $X_{VI}(\frac{1}{2},\frac{1}{2},0,0)$ is not strongly minimal! Indeed, it is not hard to see that $q=0,1,t$ define order 1 subvarieties.
\end{rem}

We define the following transformations: 

\begin{eqnarray*} 
	s_1  ( v_1, v_2, v_3 ,v_4 ) & := & ( v_2, v_1,  v_3 ,v_4 ) \\
	s_2  ( v_1, v_2, v_3 ,v_4 ) & := & ( v_1,  v_3 , v_2, v_4 ) \\
	s_3  ( v_1, v_2, v_3 ,v_4 ) & := & ( v_1,  v_2, v_4, v_3  ) \\
	s_4  ( v_1, v_2, v_3 ,v_4 ) & := & ( v_1,  v_2, -v_4, -v_3  ) \\
	s_5  ( v_1, v_2, v_3 ,v_4 ) & := & ( -v_2+1,  -v_1+1, v_3, v_4 ) \\
\end{eqnarray*}
Let $W$ be the group generated by $s_1, \ldots , s_5$. Then from the proof of Theorem 1.1 in \cite{watanabe1998birational} we have that for any $g \in W$, there is a birational bijective map $g_*$ such that the following diagram commutes:  

\[ \begin{CD}
X_{VI} @>g_*>> X_{VI}\\
@VV \pi_2 V @VV \pi_2 V\\
\m C^4 @> g  >> \m C^4
\end{CD}.\] 
It is thus the case that for any two points $\bar v, \bar w \in C^4$ which are related, $g \bar v = \bar w$, by an element of $W$, we have $X_{VI} (\bar v )$ is isomorphic to $X_{VI} ( \bar w).$ By composing the generators of $W$ given above, one can see that $g \bar v = \bar w$ for some $g \in G$ if and only if there is some $\sigma \in S_4$ and some $i,j,k,l \in \{0,1 \}$ such that $i+j+k+l \in \{0,2,4\}$ so that \[(v_1, v_2, v_3, v_4) - ((-1)^i w_{\sigma (1)}, (-1)^j w_{\sigma (2)}, (-1)^k w_{\sigma (3)}, (-1)^l w_{\sigma (4)} ) \in \m Z^4\] and 
\[v_1 - (-1)^i w_{\sigma (1)}+ v_2- (-1)^j w_{\sigma (2)}+v_3-(-1)^k w_{\sigma (3)}+v_4-(-1)^l w_{\sigma (4)} \in 2 \m Z.\] 

Notice that the orbit of any single point $\bar \alpha \in \m C^4$ under the group $W$ is Zariski dense in $\m A^4$. We now discuss the result of Watanabe with regard to model the strong minimality of $X_{VI}$.

Let $R$ be the collection of $24$ vectors of the following form: 
\[ (\pm 1 , \pm 1, 0 , 0) , (\pm 1 ,, 0  \pm 1,  0) , (\pm 1 ,0,0, \pm 1) , (0, \pm 1 , \pm 1,  0) , (0, \pm 1 ,0 , \pm 1) , (0,0, \pm 1 , \pm 1) .\] Let $\langle \bar v , \bar w \rangle = v_1   w_1 +  v_2   w_2 +  v_3   w_3 +  v_4   w_4$ denote the usual inner product on $\m C^4$. For $\alpha \in R$ and $k \in \m Z$, define \[H_{\alpha, k } = \{ \bar v \in \m C^4 \, | \, \langle \bar v , \alpha \rangle = k \}.\] 

Define
\begin{eqnarray} M & : = & \bigcup _{ \alpha \in R , \, k \in \m Z} H_ { \alpha , k} \\
P  & : = & \bigcup _{ \text{linearly ind. } \alpha, \beta  \in R , \, k, l  \in \m Z}  \left( H_ { \alpha , k} \cap  H_ { \beta , l} \right) \\
L & : = &  \bigcup _{ \text{linearly ind. } \alpha , \beta , \gamma \in R , \, k, l, m  \in \m Z} \left( H_ { \alpha , k} \cap  H_ { \beta , l} \cap H_{\gamma , m } \right)   \\
D  & : = & \bigcup _{ \text{linearly ind. } \alpha , \beta , \gamma, \delta  \in R , \, k, l, m, n  \in \m Z} \left( H_ { \alpha , k} \cap  H_ { \beta , l} \cap H_{\gamma , m } \cap H_{\delta , n} \right) 
\end{eqnarray} 

The results quoted next come from \cite{watanabe1998birational}. They depend in each case on the action of a group of affine linear transformations on the family of equations. The specific cited references derive the facts below for parameters in the fundamental domain of the group action.









\begin{prop}\cite[Theorem 2.1]{watanabe1998birational} \label{P6ranks}
\begin{enumerate} 

\item  For $\bar v \notin M$, the solution set $X_{VI} (\bar v)$ of the sixth Painlev\'e equation is strongly minimal . 

\item For $\bar v \in M \setminus P$, $X_{VI} (\bar v)$ is Morley rank one and Morley degree two. 

\item For $\bar v \in P \setminus L$, $X_{VI} (\bar v)$ is Morley rank one and Morley degree three. 

\item For $\bar v \in L \setminus D$, $X_{VI} (\bar v)$ is Morley rank one and Morley degree four. 

\item For $\bar v \in D$, $X_{VI} ( \bar v)$ is Morley rank one and Morley degree five. 
\end{enumerate} 
\end{prop}

\begin{thm} Let $\bar v \in \m C^4$ be generic over $\m Q$ and let $\bar w \in \m C^4.$ Then if $X_{VI} (\bar v )$ and $X_{VI} (\bar w) $ are nonorthongonal, it must be that after some permutation of the coordinates of $\bar w,$  $v_i=a_i w_i +c_i$ for $a_i \in \{-1,1\}$ and $c_i \in \m Z$.
\end{thm} 

\begin{proof} 
When the tuple $\bar w$ is not generic or the two tuples are not interalgebraic over $\m Q$, the argument is easier - specialize $\bar v$ to a tuple such that $X_{VI} (\bar v)$ has Morley degree not equal to the Morley degree of $X_{VI} (\bar w).$ We will omit the details of this proof as it is similar to Case I of the proof of Proposition \ref{P5ortho}. 

So, let $\bar v, \bar w \in \m C^4$ be two generic tuples over $\m Q$ which are interalgebraic. Nagloo \cite{nagloo2019algebraic} shows that for generic parameter values, $X_{VI}$ satisfies the algebraic independence conjecture. It follows that if $X_{VI} ( \bar v )$ and $X_{VI} (\bar w)$ are nonorthogonal, then there is a bijection between $X_{VI} ( \bar v )$ and $X_{VI} (\bar w)$, which by Lemma \ref{GST} can be taken to be defined over $\m Q (t, \bar v, \bar w ).$ So, let $\phi ( x,y, \bar v, \bar w, t)$ define the bijection witnessing the nonorthogonality of $X_{VI} ( \bar v )$ and $X_{VI} (\bar w)$. 

The reader can verify that for $\alpha, \beta, \gamma, \delta \in R$ linearly independent as in Proposition \ref{P6ranks} we have that the matrix $$\left( \begin{matrix}
\alpha \\
\beta \\
\gamma \\
\delta
\end{matrix}\right)$$ has determinant $\pm 2$ and $\bar v \in D$ if and only if one of the following occurs: 
\begin{enumerate} 
\item $\bar v \in \m Z^4$,
\item $\bar v \in (\frac{1}{2}+ \m Z)^4$,
\item two of the coordinates of $\bar v$ are in $\m Z$ and two of the coordinates of $\bar v \in \frac{1}{2}+ \m Z.$ 
\end{enumerate} 
Now, since we are in the case that $\bar v, \bar w$ are interalgebraic over $\m Q$, we must have polynomials $p_i$ with coefficients in $\m Q$ such that $p_i (w_i , 
\bar v) = 0$. However, as definable bijections preserve Morley degree, and the set $D$ is dense in $\m A^4,$ it must be that for all but a proper closed subset of $\bar v \in D$, if $p_i (w_i, \bar v ) = 0$, then $w_i \in D$. It follows, identically to the argument of Proposition \ref{P3orth} that $p_i$ must be linear with each variable of $\bar v$ appearing in precisely one of the $p_i$, so $p_i=w_i+a_i v_{\sigma(i)}+c_i$. Again, similarly to earlier proofs, the coefficients of $p_i$ must be $\pm 1$, and the constant term must be in $\m Z \cup (\frac{1}{2}+ \m Z).$ 

Now, for all $\bar v $ outside of a closed subset of the orbit of $(0,0,0,\frac{1}{2})$, the solution to the system given by the $p_i$ must give $\bar w$ in the orbit of $(0,0,0,\frac{1}{2})$, since these are the unique fibers of $X_{VI}$ which are nonorthogonal to a Manin kernel.\footnote{This follows from the classification of algebraic solutions of Painlev\'e six of \cite{lisovyy2014algebraic}} It follows that we must have the constant terms of the $p_i$ be integers as any system with non-integer constant terms on some $p_i$ cannot preserve the orbit $(0,0,0,\frac{1}{2})$ outside of the closed subset. 
\end{proof} 

We can in fact prove slightly more than the statement of the previous proposition - when \[w_1 + a_1 v_{\sigma (1)}+ w_2+a_2 v_{\sigma (2)}+w_3+a_3 v_{\sigma (3)}+w_4+a_4 v_{\sigma (4)} \in 2 \m Z, \] there is a known B\"acklund transformation between the fibers $X_{VI} (\bar v)$ and $X_{VI} (\bar w),$ and it follows similar to the arguments of previous sections that the bijection defined by the formula $\phi ( x,y, \bar v, \bar w, t)$ in the proof must be this B\"acklund transformation. Thus, one can establish that the classically known B\"acklund transformations are the only instances of nonorthogonality between generic fibers of Painlev\'e six by showing that we must have \[w_1 + a_1 v_{\sigma (1)}+ w_2+a_2 v_{\sigma (2)}+w_3+a_3 v_{\sigma (3)}+w_4+a_4 v_{\sigma (4)} \in 2 \m Z .\] The methods considered here are not quite enough to show this, though perhaps a very detailed analysis of the algebraic solutions \cite{lisovyy2014algebraic} would allow one to establish this fact.

\section{Degenerations of $P_{III}$} \label{DegeneP3}
The third Painlev\'e family of equations is given by \begin{equation} \label{P3og} \frac{d^2y}{dt^2} = \frac{1}{y} \left( \frac{dy}{dt} \right) ^2- \frac{1}{t} \frac{dy}{dt} + \frac{\alpha y^2 + \beta }{t}+ \gamma y^3+ \frac{\delta}{y} \end{equation} for $\alpha, \beta, \gamma, \delta$ complex numbers. In section \ref{twotwo}, the equation was put into Hamiltonian for, under the assumption that $$\delta \neq 0 , \, \gamma \neq 0.$$

In this section, we study several families of equations of the $P_{III}$ form which have not previously received analysis from the model theoretic perspective, in the process establishing severeal new transcendence results. We consider cases (our notation follows that of \cite{ohyama2006studies}): 
\begin{itemize}
\item $(D_7)$: $\gamma = 0, \, \alpha \delta \neq 0$ \emph{or} $\delta =0, \, \beta \gamma \neq 0.$
\item $(D_8)$: $\gamma = \delta =0$ and $\alpha \beta \neq 0.$
\item $(Q)$: $\alpha=0, \gamma =0$ or $\beta = 0, \delta =0.$ 
\end{itemize} 

All cases of the four parameter values for $P_{III}$ fall into the generic case considered in section \ref{twotwo} or the above cases after applying suitable transformations from \cite[Theorem 1]{ohyama2006studies}.  Most work on the third Painlev\'e equation has been in so-called generic or non-degenerate case ($\delta \neq 0, \gamma \neq 0$) which was considered in section \ref{twotwo} and from a model theoretic perspective first in \cite{nagloo2017algebraic}. In this section, we consider the other cases, with our main reference being \cite{ohyama2006studies}. Cases $D_8$ and $Q$ require very little work, and will be quickly dispensed with, but case $D_7$ requires more considerable analysis. 


\begin{rem} We record here one small note regarding the comments at the top of page 150 of \cite{ohyama2006studies}. There, the authors mention that showing $td_{\mathbb C (t) } (K) = 2$ in the case that $K$ is the differential extension of $\m C(t)$ generated by a solution of our equation is sufficient to establish the irreducibility of the equation in the sense of Umemura. This is not the case, as one has to establish this fact for more general differential fields than $\m C(t).$ Luckily, the authors do in fact accomplish this more general task later in their paper (Section 5).\end{rem}  

\subsection{Painlev\'e III - $D_8$ and $Q$}
We begin with $P_{III}$ with parameters satisfying assumption $(D_8)$ as above. In this case, \cite{ohyama2006studies} shows that each equation in the family is birationally isomorphic to equation $P_{III'}(-4,4,0,0)$, which is birationally isomorphic to a single equation in the non-degenerate family by Theorem 1, part (iii) of \cite{ohyama2006studies}. Thus the analysis from section \ref{twotwo} applies. 

In case $(Q)$, assume without loss of generality that $\beta = \delta=0$ (otherwise apply transformation 3 from Theorem 1 \cite{ohyama2006studies}). Now, the equation \ref{P3og} after the transformation $s= t^2$ and $z = ty$ is given by: 
\begin{equation} \frac{d^2z}{ds^2} = \frac{1}{z} \left( \frac{dz}{ds} \right) ^2- \frac{1}{s} \frac{dz}{ds} + \frac{\alpha z^2 }{4s^2}+ \frac{\gamma z^3}{4s^2} \label{P3Q} \end{equation}
has the family of order one subvarieties $$\left( s \frac{dz}{ds} \right) ^2 - \frac{\gamma}{4} z^4 - \frac{\alpha} {2} z ^3 = (c z)^2 $$ where $c \in \m C$ is an arbitrary constant (see page 312 of \cite{okamoto1987studies}). It follows that in this case the equation has Lascar and Morley rank two. When $c \neq 0$, a solution to the equation can be written as a rational function of $d t^ c$ where $d$ is an arbitrary constant (when $c =0,$ a solution to the equation is given by a rational function of $\log t + d$) \cite[formulas (1.4) and (1.4)']{okamoto1987studies}.

\subsection{Painlev\'e III - $D_7$}

\begin{equation} \frac{d^2y}{dt^2} = \frac{1}{y} \left( \frac{dy}{dt} \right) ^2- \frac{1}{t} \frac{dy}{dt} + \frac{\alpha y^2 + \beta }{t}+ \gamma y^3+ \frac{\delta}{y} \label{P31} \end{equation}

Performing the change of variables $s= t^2$ and $z = ty$ one obtains the following system equivalent to \ref{P31}: 

\begin{equation} \frac{d^2z}{ds^2} = \frac{1}{z} \left( \frac{dz}{ds} \right) ^2- \frac{1}{s} \frac{dz}{ds} + \frac{\alpha z^2 }{4s^2}+ \frac{\beta}{4s}+ \frac{\gamma z^3}{4s^2}+ \frac{\delta}{4z} \label{P32} \end{equation} 

By another suitable change of variables (see equation (4) page 148 of \cite{ohyama2006studies}) equation \ref{P32} can be transformed to the following form: 

\begin{equation} 
 \frac{d^2z}{ds^2} = \frac{1}{z} \left( \frac{dz}{ds} \right) ^2- \frac{1}{s} \frac{dz}{ds} + \frac{2 z^2 }{s^2}+ \frac{\beta}{4s}+ \frac{1}{z} \label{P3D7}
\end{equation} 

Following this, \cite{ohyama2006studies} investigate the Hamiltonian form of equation \ref{P3D7}: 

\begin{align}
\label{P3Ham}
\tag{$\mc H$} 
\begin{split}
    t \frac{dq}{dt} & =  2 q^2 p + \alpha _1 q + t \\
    t \frac{dp}{dt} & =  -2qp^2 - \alpha_1 p -1 
    \end{split} 
\end{align} 
where $q$ and $p$ are functions of $t$ and $\alpha_1$ is a complex parameters. When we wish to specify $\alpha_1 = c$ for some $c \in \m C,$ we write the system as $\mc H ( c).$

Let $A_1$ be the group of affine transformations generated by $s_1 (\alpha_1) = -\alpha_1$ and $ \sigma (\alpha_1) = 1-\alpha_1$. For every element $\theta \in A_1$ there is a birational bijection between the solution sets $\mc H (\alpha_1) \rightarrow \mc H (\theta (\alpha_1)).$ 

Gromak \cite{gromak1979algebraic, gromak1984reducibility} classified algebraic solutions, but we follow the notation and conventions of \cite{ohyama2006studies}. The system of equations $\mc H ( 1)$ has algebraic solutions given by the conjugates of \begin{equation} \label{P3D7alg} q(t) = -\frac{1}{2} (2t)^{\frac{2}{3}} , \, p(t) = \frac{1}{3(2t)^{\frac{2}{3}}} - \frac{1}{(2t)^{\frac{1}{3}}}\end{equation} over $\m Q$. It follows that $\mc H ( n )$ has at most $3$ algebraic solutions for all $n \in \m Z.$ There are no algebraic solutions to $\mc H(\alpha)$ for $\alpha \notin \m Z$ \cite[section 6]{ohyama2006studies}. 

From \cite[Proposition 17]{ohyama2006studies}, it follows that $\mc H (\alpha)$ satisfies Umemura's condition $(J)$ for any $\alpha \in \m C.$ In model-theoretic terms: 

\begin{thm} \label{D7sm} For any $\alpha \in \m C,$ the solution set of $\mc H (\alpha)$ is strongly minimal. 
\end{thm} 

Next, we prove the weak form of the algebraic independence conjecture for solutions of $\mc H( \alpha )$ where $\alpha$ is transcendental. The proof uses a number of deep results from the model-theoretic classification of strongly minimal sets in differentially closed fields, an overview of which one can find in \cite[Section 2]{nagloo2017algebraic}. 

\begin{thm} \label{trivialP3} The solution set of $\mc H(\alpha )$ is geometrically trivial. 
\end{thm} 
\begin{proof} Let $X(\alpha)$ denote the solution set of $\mc H(\alpha ).$ By the previous theorem, $X(\alpha)$ is strongly minimal and order two. Thus, $X(\alpha )$ is orthogonal to the constants. So, by the trichotomy theorem for differentially closed fields (see Theorem 2.8 of \cite{nagloo2017algebraic}), if $X(\alpha)$ is not geometrically trivial, then it must be the case that $X(\alpha)$ is nonorthogonal to the Manin kernel of an abelian variety $A.$ It follows from \cite[Fact 2.5 (ii)]{nagloo2017algebraic} that $A$ must be an elliptic curve which is not isotrivial (not isomorphic to an elliptic curve defined over $\m C$). In fact, it follows from \cite[Fact 4.1 part (a)]{sanchez2019isolated} that $A$ must be defined over $\m C(t)^{alg}.$ 

In this case, by the Zilber trichotomy for differentially closed fields, the strongly minimal set $X(\alpha )$ is locally modular, and since there are no $\m C(t)^{alg}$ points of $X(\alpha)$ by \cite[section 6]{ohyama2006studies}, it follows that after adjoining \emph{any} single solution, $z$ to $X(\alpha )$, the set is modular and strongly minimal.\footnote{Since there are no $\m C(t)^{alg}$ solutions or subvarieties of $X(\alpha)$, all points on the definable set have the same type over $\m C(t)^{alg}$.}

Let $A_w$ denote the elliptic curve given by the closure of $y^2 = x (x-1)(x-w)$, and let $w \in \m C(t)^{alg}$ be such that $X(\alpha )$ is nonorthogonal to $A_w$. Now, by Lemma \ref{GST} (Corollary 5.5 of \cite{GST}) we must have a definable set $Y \subset X(\alpha ) \times A_w^ \sharp $ which projects onto all but a finite subset of $X(\alpha)$ defined by a formula $\chi (u,z,\alpha)$ of size at most $n$ and all but a finite subset of $A_w^\sharp$ and $Y$ has finite fibers of size at most $k$. 

Note here that $w$ might be algebraic with $\alpha$ over $\m Q (t)$ or perhaps $w, \alpha$ are independent over $\m Q(t)$. We will obtain a contradiction in either case:
\begin{itemize} 
\item Case I: $(w, \alpha)$ is a zero of an irreducible polynomial $p(w,\alpha)$ over $\m Q(t)$.
\item Case II: $w \in \m Q(t )^{alg}$. 
\end{itemize} 
. Let $\phi (u, (x,y) , z , \alpha, w)$ define the correspondence $Y$. 

Define property $\star_{\alpha,1}$ to hold if we have that for all $z \in X(\alpha),$ $\phi (u, (x,y) , z , \alpha, w)$ is a finite-to-finite correspondence with fibers of size at most $k$ between $X(\alpha) \setminus \chi (u,z,\alpha)$ and all but at most $n_1$ elements of $A_w^\sharp$ for $w$ such that $p(w,\alpha)=0$, where $\chi (u,z,\alpha) \subset X(\alpha)$ defines a subset of size at most $n$.

If we are in Case I, property $\star_{\alpha,1}$ is a first-order property of $\alpha$ over $\m Q(t)$, and it holds for generic $\alpha \in \m C.$ The definable set $v'=0$ is strongly minimal, and so it must be that $\star_{c , 1}$ holds for all but finitely many $c \in \m C.$ 

So, there is some $n \in \m Z$ such that the property holds. Now, $\phi (u, (x,y) , z , n, w)$ defines a finite-to-finite correspondence between $X(n) \setminus \chi (u,z,n)$ and all but at most $n_1$ elements of $A_w^{\sharp}$. Choose $z$ to be a solution of $X(n)$ given by equation \ref{P3D7alg} above. Now there is a finite-to-finite correspondence defined over $\m C(t)^{alg}$  between a Manin kernel $A_w^{\sharp}$ of an elliptic curve over $\m C(t)^{alg}$ and $X(n)$ excepting finitely many points.

But $A_w^\sharp $ has infinitely many $\m C(t) ^{alg}$ points and definable finite-to-finite correspondences over $\m C(t)^{alg}$ preserve that the points are in $\m C(t)^{alg}.$ Thus $X(n)$ has infinitely many $\m C(t)^{alg}$ points, but this contradicts \cite[section 6]{ohyama2006studies}. 

Case II follows by a very similar argument after defining property $\star_{\alpha,2}$ to hold if we have that for all $z \in X(\alpha),$ $\phi (u, (x,y) , z , \alpha, w)$ is a finite-to-finite correspondence with fibers of size at most $k$ between $X(\alpha) \setminus \chi (u,z,\alpha)$ and all but at most $n_1$ elements of $A_w^\sharp$ where $\chi (u,z,\alpha) \subset X(\alpha)$ defines a subset of size at most $n$.\footnote{Note that $\star_{\alpha, 2}$ is the same as $\star_{\alpha, 1}$ except for the vanishing of polynomial $p.$} The remainder of the argument is the same as Case I, but here we leave $w$ fixed as we specialize $\alpha.$ 
\end{proof} 

\begin{rem} 
Our proof of Theorem \ref{trivialP3} can be readily adapted to give proofs of the triviality of Painlev\'e equations with generic coefficients of each of the other families, a result originally proved in \cite{nagloo2017algebraic} via a slightly different method which exploited the Riccati subvarieties of the families (these Riccati subvarieties correspond to classical solutions to equations in the Painlev\'e families for certain special values of the parameters). Adapting the technique above to any of the other Painlev\'e families is fairly straightforward relative to the classification of algebraic solutions of those families, which we used extensively in the previous sections. Our approach also uses \cite[Fact 4.1 part (a)]{sanchez2019isolated} which Le\'on-S\'anchez and Moosa mention was not widely known previously (though it can be deduced from results of \cite{HrSo}). The technique of \cite{nagloo2017algebraic}, however,  will not work in our case, since our family, $P_{III}$ in the $D_7$ case, has no Riccati subvarieties. 
\end{rem} 

Next, we prove that for transcendental $\alpha,$ $\mc H(\alpha )$ is $\omega$-categorical, the analog of results established in \cite{nagloo2014algebraic} for other Painlev\'e families. 


\begin{thm} \label{catP3} The solution set of $\mc H(\alpha )$ is $\omega$-categorical. 
\end{thm} 

\begin{proof}
We refer the reader to \cite{nagloo2017algebraic, nagloo2014algebraic} for a discussion of the property of $\omega$-categoricity, which in this contexts amounts to showing that over any finitely generated differential field extension $L$ of $\m Q(\alpha , t)$ given any solution $x$ of $\mc H( \alpha )$, there are only finitely many solutions of $\mc H (\alpha)$ which are interalgebraic with $x$ over $L.$ In fact, we will show that for any solution $x $ of $\mc H(\alpha)$, there are at most two other solutions which are interalgebraic with $x$. 

By remark 2.4 of \cite{nagloo2014algebraic}, it is sufficient to take $L = \m Q (\alpha, t).$ Now, assume that there are at least $3$ other distinct solutions of $\mc H(\alpha )$, $y_1, y_2, y_3$, which are interalgebraic with $x$ over $\m Q (\alpha, t ).$ 

Then, $\mc H( \alpha )$ has the property that for any solution $x$, there are at least $3$ additional distinct interalgebraic solutions. This property is expressible by a first order formula depending on $\alpha,$ and thus it must be true of cofinitely many $c \in \m C.$ In particular, it must be true of some $n \in \m Z.$ Then taking $x$ to be one of the three algebraic solutions of $\mc H(n)$ yields a contradiction, since over $\m Q (t)$ there are only $3$ algebraic solutions of $\mc H(n ).$ 
\end{proof} 

Examining the proof of Theorem \ref{catP3}, we can see that any solution of $\mc H(\alpha)$ might be interalgebraic with at most two other solutions. We conjecture there are \emph{no algebraic relations} between solutions of $\mc H(\alpha)$. Proving this conjecture, by the above results, requires only answering the following question affirmatively. 
\begin{ques} \label{quesinddeg}
For $\alpha \in \m C$ transcendental and $x_1, x_2$ solutions of $\mc H (\alpha)$, are $x_1, x_2$ algebraically independent over $\m Q ( \alpha , t)$?   
\end{ques}

\subsection{Orthogonality and $P_{III}$ - $D_7$} 

In this subsection, we analyze the $D_7$ family with respect to the nonorthogonality relation, both with respect to other families of Painlev\'e equations and its own fibers. 

\begin{prop}
Suppose that $\alpha$ is transcendental over $\m Q$. If $\mc H(\alpha )$ is nonorthogonal to $\mc H( \beta),$ then $\alpha + \beta \in \m Z$ or $\alpha -\beta \in \m Z$. 
\end{prop}
\begin{proof} Let $\alpha$ be transcendental over $\m Q$. By Theorems \ref{D7sm} \ref{trivialP3} \ref{catP3}, $\mc H(\alpha)$ is strongly minimal, trivial, and $\omega$-categorical; even more, the proof of \ref{catP3} shows that any solution of $\mc H(\alpha)$ is algebraically independent from all but at most two other solutions of $\mc H(\alpha)$. We will also use the fact that $\mc H (\alpha)$ has no solutions in $\m C(t)^{alg}$ (section 6 of \cite{ohyama2006studies}). 

Suppose that $\mc H(\alpha)$ is nonorthogonal to $\mc H(\beta)$. Then by Lemma \ref{GST} there must be a generically finite-to-finite definable relation between $\mc H(\alpha)$ and $\mc H(\beta)$ given by a formula $\phi (x,y , \alpha, \beta , t)$ over the emptyset. Since $\mc H(\alpha)$ has no algebraic solutions, $\phi $ must be onto $\mc H(\alpha).$ We now consider two cases. 

\begin{enumerate} 
\item Case I: $\alpha$ is algebraically independent from $\beta$. 
\item Case II: $\alpha$ and $\beta $ are interalgebraic. 
\end{enumerate} 

Case Ia: $\beta \notin \m Z$. The formula $\phi (x,y , \alpha, \beta , t)$ gives a generically finite-to-finite correspondence which is onto $\mc H(\alpha)$ for a Zariski open subset of $\alpha \in \m C$. So, consider some $n$ in this open subset. Then $\phi (x,y, n , \beta , t)$ gives a finite-to-finite relation between $\mc H( n ) $ and $\mc H( \beta).$ But then the solutions to $\mc H( \beta )$ which are related to the algebraic solutions of $\mc H( n)$ must be algebraic, since the formula $\phi$ has finite fibers. This is a contradiction, since $\mc H(\beta )$ has no algebraic solutions. 

Case Ib: $\beta \in \m Z.$ The formula $\phi (x,y , \alpha, \beta , t)$ gives a generically finite-to-finite correspondence which is onto $\mc H(\alpha)$ for a Zariski open subset of $\alpha \in \m C$. First, suppose that this correspondence is onto $\mc H( \beta)$; this is impossible, since $\mc H( \beta )$ has algebraic solutions and $\mc H (\alpha ) $ does not. Thus, the algebraic solutions of $\mc H(\beta )$ must lie outside the domain of the relation $\phi$. Now, specialize $\alpha$ to some $n$ in the above open set. But now the algebraic solutions of $\mc H( n)$ must be related to nonalgebraic solutions of $\mc H(\beta )$, a contradiction. 

Next, we consider case II: $\alpha$ and $\beta$ are interalgebraic over $\m Q$. So, let $p(\alpha, \beta )=0$ define the affine irreducible curve over $\m Q^{alg}$ on which $\alpha, \beta $ is a generic point. Then since $\mc H(\beta )$ has no solutions in $\m Q(t)^{alg}$, the relation given by $\phi$ must be onto $\mc H(\beta)$. Thus, $\phi (x,y, \alpha, \beta , t)$ gives a finite-to-finite relation onto both $\mc H(\alpha )$ and $\mc H(\beta )$ for a dense open set of $(\alpha , \beta )$ in the curve $C.$ Since we can not have such a correspondence between $\mc H(n)$ for $n \in \m Z$ and $\mc H(a)$ for $a \notin \m Z,$ it must be the case that for all but finitely many integers $n$, $p(n,y)=0$ has only integer solutions and $p(x,n)$ has only integer solutions. Now, it follows by the same arguments as in \ref{genone} that $p(x,y)$ must be given by $x-y-n$ or $x+y+n$ for some $n \in \m Z.$ 

\end{proof} 

A positive answer to question \ref{quesinddeg} would allow one to (by similar methods to \ref{genone}) show that the relations witnessing nonorthogonality of $\mc H(\alpha)$ and $\mc H (\beta )$ in the previous result must be given by the known B\"acklund transformations - see the beginning of this section and \cite{ohyama2006studies} for details. In general, as with the other families of Painlev\'e equations, we conjecture that the known B\"acklund transformations give all instances of nonorthogonality within the family.  

\begin{thm} Equation $\mc H$ with generic coefficient $\alpha$ is orthogonal to all of the other Painlev\'e equations with generic coefficients. 
\end{thm} 

\begin{proof} Let $X$ be one of the families of $P_{II}$ through $P_{VI}$, and let $\beta$ be a generic parameter (tuple of parameters). Suppose that $X(\beta)$ is nonorthogonal to $\mc H(\alpha )$; by Lemma \ref{GST}, there must be a generically finite-to-finite relation defined over $\m Q(t, \alpha, \beta)$ between $\mc H( \alpha )$ and $X(\beta )$. Neither $X(\beta )$ nor $H ( \alpha )$ have algebraic solutions, and there are no algebraic relations between solutions of $X(\beta )$, so it follows that the relation must be a bijection $X(\beta ) \rightarrow \mc H ( \alpha )$ given by a formula $\phi (x, t, \beta , \alpha)$. The formula $\phi (x,t,\beta , \alpha)$ gives a bijection between the solution sets $X(\beta ) \rightarrow \mc H ( \alpha )$ for all $\beta$ in some Zariski open set. In each of the Painlev\'e families, such a set includes fibers of Morley degree greater than one, while $\mc H$ has no such fibers, a contradiction since definable bijections preserve Morley degree. 
\end{proof}

\section{Besides algebraic relations.} We would also like to mention a few other aspects of our work which may be of interest from the point of view of model theory or differential algebraic geometry. Our results show give the first examples of a positive answer to the following question:

\begin{ques} \label{degcont} Is there a family of differential varieties $X \rightarrow Y$ in which for some $n \in \m N $, $$\{ y \in Y \, | \, d_M (X_y) =n \} $$ is not definable? Here $d_M(-)$ denotes the Morley degree.   
\end{ques}
Having \emph{Morley degree one} is a natural notion of \emph{irreducibility} for definable sets, but the Kolchin topology already comes equipped with its own notion of irreducibility:  a Kolchin-closed set $X \subset \m A^n$ over $K$ is irreducible if the collection of differential polynomials with coefficients in $K$ which vanish on $X$ forms a prime differential ideal. 

\begin{ques} \label{ritt} Is there a family of differential varieties $X \rightarrow Y$ in which $$\{ y \in Y \, | \, X_y \text{ is irreducible} \} $$ is not definable? 
\end{ques} 
Question \ref{ritt} is one of the equivalent forms of the \emph{Ritt problem}, an important open problem which has received considerable attention, e.g. \cite{golubitsky2009generalized}. Questions \ref{ritt} and \ref{degcont} are of a similar flavor, but the interaction of model-theoretic ranks with the Kolchin topology is somewhat enigmatic; for instance, \cite{freitag2012model} gives an example of a definable set whose Kolchin closure has higher Lascar and Morley rank than the original set. 

Finally, we discuss definability and orthogonality. In \cite{HrIt}, Hrushovski and Itai show that if $X$ is any order one strongly minimal strictly disintegrated set, and $Y \rightarrow B$ is a family of definable sets, then $\{b \in B \, | \, Y_b \not \perp X \}$ is a definable set. When $X$ is taken to be the definable subfield of the constants, the set of fibers of a family nonorthogonal to $X$ is not necessarily definable, a fact which figures prominently in a number of works, e.g. \cite{mcgrail2000search}. Every order one set is nonorthogonal to a strictly disintegrated set or the constants \cite{freitag2017finiteness}. 

Our analysis of the Painlev\'e families shows that the analog of the Hrushovski-Itai result does not hold for order two strictly disintegrated sets; the nonorthogonality classes of generic Painlev\'e equations are given by the orbits of discrete groups. This seems to be the first instance in which such a phenomenon has been noticed for trivial strongly minimal sets. A similar phenomenon was already noticed by Hrushovski and Sokolovic~\cite{HrSo}  for locally modular sets; every locally modular strongly minimal set is nonorthogonal to the Manin kernel $A^\sharp$ of a simple Abelian variety $A$ which does not descend to the constants. Two such sets $A^ \sharp$ and $B^\sharp$ are nonorthogonal precisely when $A$ and $B$ are isogenous. Considering the case of elliptic curves for simplicity, $A$ and $B$ are isogenous precisely when their $j$-invariants satisfy a modular polynomial. Our results show a similar structure for nonorthogonality for Painlev\'e equations with generic parameters. 

As we have seen, instances of nonorthogonality have a natural interpretation as functional transcendence results, but they have also played a key role in structural problems concerning differentially closed fields, a theme we have not touched on. For instance, one can use the nonorthgonality structure of any of the Painlev\'e families to give a simple proof that the isomorphism problem for countable differentially closed fields is Borel complete using the ENI-DOP style arguments of \cite{marker2007number}.

\bibliography{Research}{}
\bibliographystyle{plain}

\end{document}